\documentclass[
onefignum,onetabnum]{siamonline250211}

\usepackage{amsmath,amsfonts,amssymb,amsbsy}
\usepackage{mathrsfs}           
\allowdisplaybreaks

\newtheorem{assumption}[theorem]{Assumption}

\newcommand{\sgn}{\operatorname{sgn}}

\newcommand{\norm}[1]{\| #1 \|}
\newcommand{\braket}[1]{\left\langle #1\right\rangle}

\newcommand{\bbE}{\mathbb E}
\newcommand{\bbP}{\mathbb P}
\newcommand{\bbR}{\mathbb R} 
\newcommand{\bbS}{\mathbb S}

 \newcommand{\cD}{\mathcal D}

\newcommand{\cI}{\mathcal I} \newcommand{\cJ}{\mathcal J}
\newcommand{\cK}{\mathcal K} 
 \newcommand{\cN}{\mathcal N}
\newcommand{\cO}{\mathcal O}

\newcommand{\cU}{\mathcal U} \newcommand{\cV}{\mathcal V}

\newcommand{\D}{{\mathcal D}}

\newcommand{\hiSRD}{{\sc hiSRD}}

\renewcommand{\epsilon}{\varepsilon}

\usepackage{tikz}
\usepackage{pgfplots}
\usepackage{pgfplotstable}
  \pgfplotsset{compat=newest}
\pgfplotsset{table/search path={../}}
\usepgfplotslibrary{fillbetween}
\usepackage{siunitx}
\graphicspath{{../}}

\usepackage[sort,nocompress]{cite}

\usepackage{enumitem}

\newsiamremark{remark}{Remark}
\newsiamremark{example}{Example}

\title{Infinite-dimensional spherical-radial decomposition for
probabilistic functions,
with application to constrained optimal control and Gaussian process regression\thanks{Partially supported by the U.S.\ Department of Energy, Office of Science Energy Earthshot Initiative under Award \#DE-SC0024721, and by the US National Science Foundation (NSF) under \#2411229.
}}

\author{Kewei Wang\thanks{Courant Institute School of Mathematics, Computing and Data Science, New York University, New York, USA, \email{kw3645@nyu.edu}, \email{stadler@cims.nyu.edu}} \and  Georg Stadler\footnotemark[2]}

\date{\today}

\begin{document}
\maketitle

\begin{abstract}
The spherical-radial decomposition (SRD) is an efficient method for estimating probabilistic functions and their gradients defined over finite-dimensional elliptical distributions. In this work, we generalize the SRD to infinite stochastic dimensions by combining subspace SRD with standard Monte Carlo methods. The resulting method, which we call hybrid infinite-dimensional SRD (\hiSRD) provides an unbiased, low-variance estimator for convex sets arising, for instance, in chance-constrained optimization. We provide a theoretical analysis of the variance of finite-dimensional SRD as the dimension increases, and show that the proposed hybrid method eliminates truncation-induced bias, reduces variance, and allows the computation of derivatives of probabilistic functions. We present comprehensive numerical studies for a risk-neutral stochastic PDE optimal control problem with joint chance state constraints, and for optimizing kernel parameters in Gaussian process regression under the constraint that the posterior process satisfies joint chance constraints.
\end{abstract}



\section{Introduction}
This work addresses optimization problems under uncertainty, specifically the approximation of a class of probabilistic functions with infinite-dimension\-al underlying random variables. Such probability functions arise, for example, in risk-averse stochastic optimization, including value-at-risk objectives and chance-constrained optimization.

Most stochastic optimization problems common in economics, operations research, or statistics use a moderate number of stochastic variables and fast-to-evaluate (although often strongly nonlinear) objectives. Stochastic optimization governed by infinite-dimensional constraints, such as partial differential equations (PDEs), introduces additional difficulties. These stem from the infinite-dimensional nature of the operators, the resulting high-dimensional problems after discretization, and the inherently large or even infinite stochastic dimensions.
Analyzing the theoretical behavior of these problems and constructing approximations requires synthesizing methods from several fields, including stochastic optimization, uncertainty quantification, PDE discretization and numerical solvers, and PDE-constrained optimization \cite{HeKo-Acta25}.

Specifically, we consider stochastic optimization problems with probabilistic constraint of the following form:
\begin{equation}\label{eq:Jopt}
        \min_{u\in U_{\mathrm{ad}}} \quad  \cJ(u) \quad 
        \mathrm{s.t.} \quad   \varphi(u)\geq p,
\end{equation}
where $\cJ:U_{\text{ad}}\mapsto \bbR$ is the objective, $u$ is the optimization variable, $p\in [0,1]$, and the probability function is defined as
\begin{equation}\label{Eq:phi_u}
    \varphi(u):=\bbP(g(u,{\xi})\leq 0),
\end{equation}
where $g:U_\text{ad}\times \mathscr{H}\to \mathbb{R}$ is convex (and typically non-smooth) in $\xi$, which is an elliptical random variable taking values in a Hilbert space $\mathscr H$. A primary challenge addressed in this work arises when $\xi$ is a function on a spatial domain $\D\subset \mathbb R^n$ that requires an \emph{infinite} series:
\begin{equation}\label{eq:inf-xi}
\xi(x) = \bar\xi(x) + \sum_{k=1}^\infty \xi_k \phi_k(x), \quad x\in \D, 
\end{equation}
where $\phi_k(\cdot)$ are fixed functions and the coefficients $\xi_k$ are scalar random variables. For instance, \eqref{eq:inf-xi} may occur from a Karhunen-Lo\`eve (KL) expansion. Understanding the properties of \eqref{Eq:phi_u} and efficiently approximating both $\varphi(u)$ and its gradient $\nabla \varphi(u)$ are essential to solve \eqref{eq:Jopt}. In this work, we develop methods to approximate \eqref{Eq:phi_u} and its gradient for variables $\xi$ of the form \eqref{eq:inf-xi}, for example, when $\xi$ is a Gaussian random field. In \cref{sec:intro-examples}, we present two concrete examples of \eqref{Eq:phi_u} arising in chance-constrained stochastic optimization.

\subsection{Approach}
Evaluating \eqref{Eq:phi_u} requires high-dimensional integration over a set, which is usually done using Monte Carlo methods. For \emph{finite-dimensional} elliptical distributions and sets that arise, for example, in chance-constrained optimization, transforming the random variable into spherical coordinates proves very useful. Approximating $\varphi(u)$ then amounts to performing the integration by first uniformly sampling points on the unit sphere and then integrating exactly along the resulting radial directions.
The resulting approach, called spherical-radial decomposition (SRD), allows the computation of gradients of $\varphi(\cdot)$, enables the use of low-discrepancy Monte Carlo sequences on the unit sphere, and often results in a strictly lower variance estimator compared to the standard Monte Carlo sampling \cite{AcHe-SIAM14,AcHe-SIAM17,HeStWe-SIAM25}.

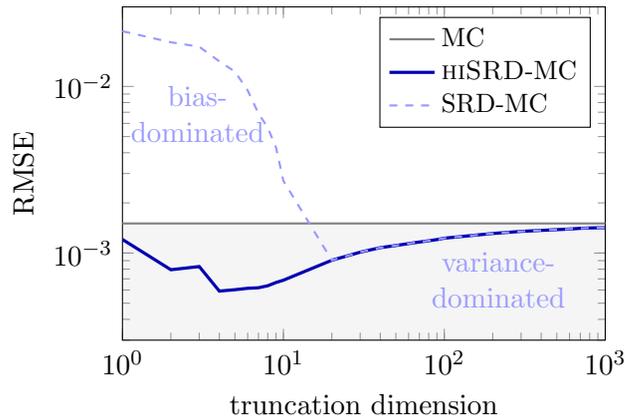
\begin{figure}[tb]
    \centering
    \begin{tikzpicture}
        \begin{loglogaxis}[
            xlabel={truncation dimension},
            ylabel={RMSE},
            xmin=1, xmax=1000,
            ymin=3e-4, ymax=3e-2,
            width=8cm, height=6cm,
            legend pos=north east,
            legend cell align={left},
            legend style={font=\small},
            grid style=dotted,
        ]
        \addplot[name path=MC, color=black!50!white, mark=none, thick]
            table[x=k, y=MC] {data/hiSRD.txt};
            \addlegendentry{MC}
        \path[name path=axis] (axis cs:1,1e-4) -- (axis cs:1000,1e-4);
        \addplot[black, opacity=0.04, on layer=axis background, forget plot] fill between[of=MC and axis];   
        \addplot[color=black!30!blue, mark=none, very thick]
            table[x=k, y=hiSRD_MC] {data/hiSRD.txt};
            \addlegendentry{\hiSRD-MC}            
        \addplot[color=blue!40, mark=none, thick, dashed]
            table[x=k, y=SRD_MC] {data/hiSRD.txt};
            \addlegendentry{SRD-MC}
        \end{loglogaxis}
        \node[color=blue!40,align=center] at (1,3) {bias-\\dominated};   
        \node[color=blue!40,align=center] at (5,0.8) {variance-\\dominated};   
    \end{tikzpicture}
    \caption{Comparison of root mean squared error (RMSE; $y$-axis) versus stochastic truncation dimension ($x$-axis) for standard Monte Carlo (gray horizontal line), finite-dimensional SRD (dashed) and for the proposed method (solid). Results are for Gaussian process example detailed in \cref{SEC:NumGPR}, but with 1000 repeats to estimate the RMSE.
    }\label{fig:intro}
\end{figure}

When applying SRD to \emph{infinite-dimensional} random variables, such as Gaussian random fields, the standard approach relies on truncating the random variable expansion. However, this truncation introduces a bias. While increasing the truncation dimension reduces bias, the resulting impact on the SRD estimator is not well understood. In this work, we analyze how high stochastic dimensions deteriorates the variance reduction effect of SRD, and propose a hybrid approach combining subspace SRD with Monte Carlo sampling. This hybrid formulation eliminates truncation-induced bias while preserving the differentiability and variance-reduction characteristics of SRD. The typical behavior of the resulting estimator is shown in
in \cref{fig:intro}. As can be seen, the proposed hybrid estimator remains unbiased regardless of the truncation space dimension and consistently outperforms standard Monte Carlo methods, which additionally do not provide gradients.

\subsection{Related literature}

Various methods have been proposed to approximate probabilistic functions of the form \eqref{Eq:phi_u}, particularly those originating from (joint) chance constraints. Besides the SRD, which is the main focus of this work, proposed methods include kernel density estimation \cite{CaCeSaTrZi-OCAM18}, sample average approximations and mixed-integer programming
\cite{LuShNe-MP10, PaAhSh-JOTA09},
and feasible set approximation and smoothing methods \cite{NeSh-SIAM07, HoYaZh-OR11, PeLuWa-SIAM20}.

The work most closely related to this paper focuses on stochastic, chance-constrained optimization governed by PDEs with uncertain variables \cite{FarshbafShakerGugatHeitschEtAl20, HeStWe-SIAM25, geiersbach2025numerical, geletu2020}.
In \cite{FarshbafShakerGugatHeitschEtAl20}, the authors study the control of the linear wave equation with uncertain initial data represented as an infinite Fourier series. Their numerical implementation truncates this series to permit using a finite-dimensional SRD; the authors also address how this truncation may be adjusted in relation to the accuracy of the PDE discretization. In \cite{geiersbach2025numerical, HeStWe-SIAM25}, the authors study stochastic optimization problems governed by elliptic PDEs and constrained by joint chance state constraints. In these works, the PDE solution depends affinely on the stochastic variable, which allows the use of SRD in the numerical implementation. Because the stochastic variables are expressed as sums of finitely many terms, a finite-dimensional SRD is employed. In \cite{geletu2020}, the authors study elliptic PDEs with finite-dimensional stochastic variables under individual chance constraints, which are typically more tractable than joint chance constraints.

Our problem formulation is also relevant to Gaussian process regression, where kernel functions usually depend on hyperparameters that are tuned based on observed data. If one further requires the posterior process to satisfy pointwise bound constraints jointly with a certain probability, the resulting optimization problem can be cast as \eqref{eq:Jopt}, \eqref{Eq:phi_u}, where the optimization is now governed by Gaussian process regression rather than a PDE. For a general overview of constrained Gaussian processes, we refer to~\cite{SwGuFrSaJa-JMLMC20}. For approaches to incorporate typically individual (rather than joint) pointwise bound constraints within process regression, or to sample from truncated Gaussian distributions, we refer to~\cite{PeYaZh-TAML20, KoYa-25, VeMa-12}.

\subsection{Contributions and limitations}
In this work, we make the following contributions.
(1) We generalize the SRD to infinite-dimensional distributions, resulting in an unbiased, low-variance estimator for convex sets arising in chance-constrained optimization.
(2) We provide a theoretical analysis of the variance of finite-dimensional SRD as the stochastic dimension increases.
(3) We generalize existing differentiability results for finite-dimensional SRD to the proposed method.
(4) We present comprehensive numerical studies for an elliptic PDE control problem with joint chance state constraints, and for Gaussian process regression with joint chance constraints.

The proposed approach also has some limitations.
(1) The SRD framework is limited to elliptical distributions, such as multivariate Gaussian, Laplace, or Student's $t$-distributions.
(2) The sets whose probability we seek must be convex, and are typically the intersection of (possibly infinitely many) half-spaces.
(3) While we show how to lower variance and eliminate truncation-induced bias, the proposed approach still relies on (quasi) Monte Carlo sampling and therefore has the typical Monte Carlo convergence rates.

\subsection{Examples}\label{sec:intro-examples}
We present two examples that fit the form \eqref{eq:Jopt}, \eqref{Eq:phi_u} and \eqref{eq:inf-xi}. These examples are studied numerically in~\cref{SEC:NumCtrl,SEC:NumGPR}.

\begin{example}[PDE optimal control] \label{ex:optcont}
We consider a risk-neutral stochastic optimal control problem governed by an elliptic PDE:
\begin{equation}\label{eq:ex:opt-contr}
    \begin{aligned}
        \min_{u\in U_{\mathrm{ad}}} \cJ(u), \quad \text{ where } \cJ(u) := \mathbb E \left[\tilde\cJ(u,y,\xi)\right].   
    \end{aligned}
\end{equation}
Here, $u$ is the control (typically in a closed, convex subset $U_\mathrm{ad}$ of square-integrable functions), and the random variable $\xi$ takes values
in a Hilbert space $\mathscr H$. The state variable $y(\xi)$ and the control $u$ satisfy the equation
\begin{equation}\label{eq:redstate}
A y(\xi) + B \xi + Cu  = f,
\end{equation}
where $A:Y\mapsto H^{-1}(\D)$ with $Y$ a subspace of $H^1(\D)$ is a linear invertible elliptic operator,
$B:\mathscr
H\mapsto H^{-1}(\D)$ and $C:U_{\!\text{ad}}\to L^2(\D)$ are bounded linear operators,
and $f\in L^2(\D)$. In \eqref{eq:ex:opt-contr}, the state variable $y$ is considered a function of $u$ and $\xi$ through the solution of \eqref{eq:redstate}.
Joint chance state constraints with lower and upper bounds $\underline y$ and $\overline y$ mean that
\begin{equation}\label{eq:chance}
  \mathbb P(\omega\in \Omega \,| \,\underline y(x) \le y(x, \xi,u)\le \overline y(x)
  \text{ for a.\,a.\ } x\in \D) \ge p.
\end{equation}
Variants of this problem have been studied in \cite{HeStWe-SIAM25,geiersbach2023optimality,geiersbach2025numerical,Kouri2023}, and this formulation fits the form \eqref{eq:chance} with $g(u,\xi) := \max_{x\in\D} \{\underline y(x) - y(x;\xi), y(x;\xi)-\overline y\}$\footnote{Since we do not assume continuity of $y$, the $\max$-function must be properly interpreted as detailed in \cite{HeStWe-SIAM25}.}, where $y(\omega)$ is considered a function of the control $u$ through the solution of PDE \eqref{eq:redstate}.
\end{example}

\begin{example}[Gaussian process regression]\label{ex:GPR}
Assume that we desire, by adjusting the hyperparameters of the kernel, that the realizations of the Gaussian process arising from regression satisfy bound constraints with a given high probability \cite{SwGuFrSaJa-JMLMC20, PeYaZh-TAML20}. Consider a (prior) Gaussian process $\xi_\text{pr}\sim \mathcal N(\xi_0,\mathcal K_u)$ over domain $\D$ taking values in a Hilbert space $\mathscr{H}$, with a covariance kernel $\mathcal K_u$ that depends on to-be-determined kernel parameters $u$, such as a length scale, variance and exponent. Suppose that we have observations $\boldsymbol{y}=(y^{(1)},\dots,y^{(N)})\in\bbR^N$ at locations $\boldsymbol{X}=(x^{(1)},\dots,x^{(N)})\in\cD^N$ and use a Gaussian process to fit these data. To tailor the kernel function, we minimize the negative log-likelihood as a function of the kernel parameters, i.e.,
\begin{equation}\label{Eq:ObjFunc}
    \min_{u}\quad \cJ(u):=\frac{1}{2}\left( (\boldsymbol y-\boldsymbol \xi_0)^T K_u^{-1} (\boldsymbol y-\boldsymbol\xi_0) + \log |K_u| + N\log(2\pi) \right),
\end{equation}
where $\boldsymbol\xi_0=(\xi_0(x^{(1)}),\ldots,\xi_0(x^{(N)}))$ and $K_u=\cK_u(\boldsymbol{X},\boldsymbol{X})\in\bbR^{N\times N}$ is the covariance matrix evaluated at $\boldsymbol{X}$. We require that the posterior Gaussian process $\xi_\text{post}\sim\cN(\xi^\star(u),\cK^\star_u)$ satisfies a joint chance constraint with lower and upper bounds $\underline\xi$ and $\overline{\xi}$ with probability $p\in(0,1)$:
\begin{equation}\label{Eq:ChanceConstrGPR}
    \varphi(u):=\mathbb P(\omega|\underline\xi(x)\le \xi_{\text{post}}(x,\omega;u)\le \overline\xi(x) \text{ for a.a.~} x\in \D) \ge p.
\end{equation}

To write \eqref{Eq:ChanceConstrGPR} in the form \eqref{Eq:phi_u}, we choose a reference centered Gaussian process $\xi_{\mathrm{ref}}:=\xi_{\mathrm{post}}(u_0)-\xi^\star(u_0)$ for some fixed kernel parameters $u_0$. Since $\xi_{\mathrm{post}}(u)$ takes values in $\mathscr{H}$, the posterior with different kernel parameters can be recovered by the affine transformation 
\begin{equation}\label{Eq:GPtransform}
    \xi_{\text{post}}(u)=A(u)\xi_{\mathrm{ref}}+\xi^\star(u),
\end{equation}
where the operator $A(u):\mathscr{H}\to \mathscr{H}$ satisfies $A(u) \cK_{u_0}^\star A(u)^* = \cK_u^\star$,
with the $*$ denoting the adjoint operator.
The constraint \eqref{Eq:ChanceConstrGPR} fits the form of \eqref{Eq:phi_u}, with
$g(u,\xi_{\mathrm{ref}})=\max_{x\in\cD}{\left\{ \left( \underline{\xi}(x)-\xi_{\text{post}}(x;u) \right), \left( \xi_{\text{post}}(x;u)-\overline{\xi}(x) \right)\right\}}$.

\end{example}

\subsection{Notation}
We use the abbreviation SRD for the existing spherical-radial decomposition in finite dimensions. Our extension to infinite stochastic dimension is called \hiSRD, short for \emph{hybrid infinite-dimensional spherical radial decomposition}.
To distinguish vectors from scalars or scalar functions, we use a bold font. We generally denote the spatial domain by $\cD$ and use $\Omega$ for the random space. 
When we use Karhunen-Lo\`eve expansions to describe random variables, we use $K$ for the number of expansion terms, and use $N$ for the number of samples in a Monte Carlo approximation. Upon discretization, we use $M$ to denote the number of points in $\cD$ where the state constraints are evaluated by.

\section{Dimension dependency of spherical-radial decomposition}
In this section, we summarize the SRD in finite dimensions, discuss its properties for estimating probability functions, and prove that its variance reduction properties deteriorate with increasing dimension.

\subsection{Review of finite-dimensional SRD}\label{subsec:finite-dim-SRD}
The SRD is an efficient method for estimating probabilistic functions with elliptical distributions, and it can be applied to optimization problems of the form~\eqref{eq:Jopt}. An elliptical random vector $\boldsymbol{\zeta} \in \bbR^n$ admits the decomposition
$\boldsymbol{\zeta} = \bar{\boldsymbol{\zeta}} + \tau L_K \boldsymbol{\nu}_K$,
where $\bar{\boldsymbol{\zeta}} \in \bbR^n$, $L_K \in \bbR^{n \times K}$, $\boldsymbol{\nu}_K \sim \cU(\bbS^{K-1})$ is a random vector uniformly distributed on the unit sphere, and $\tau\geq 0$ is a scalar radial random variable independent of $\boldsymbol{\nu}_K$. This decomposition is known as the SRD of $\boldsymbol{\zeta}$. Although the stochastic dimension $K$ is typically chosen as a truncation dimension for low-rank approximations of $\boldsymbol{\zeta}$,
it is not necessarily bounded by the spatial dimension $n$, and we permit over-parameterized representations where $K>n$.

While the SRD is valid for any elliptical distribution by correspondingly selecting the distribution of $\tau$, here we focus on multivariate Gaussian distributions. Specifically, consider $\boldsymbol{\zeta} \sim \cN(\bar{\boldsymbol{\zeta}}, \Sigma)$, with a possibly degenerate covariance $\Sigma=L_K L_K^T$. In this case, the radial component $\tau$ is distributed as $\chi_K$, the Chi distribution with $K$ degrees of freedom. The probabilistic function we want to estimate is
\begin{equation}\label{Eq:probfunc}
    \varphi(u) := \bbP(g(u,\boldsymbol{\zeta})\leq 0),
\end{equation}
where $g:U_{\mathrm{ad}}\times \bbR^n\to\bbR$, and $g(u,\boldsymbol{\zeta})$ is convex in $\boldsymbol{\zeta}$ for all $u\in U$. Using SRD, $\varphi(u)$ 
can be represented as
\begin{equation}\label{eq:phi(u)}
    \varphi(u)=\int_{\boldsymbol{v}\in\bbS^{K-1}}\mu_{\chi_K}\left( \left\{ r\geq 0: g(u,\bar{\boldsymbol{\zeta}}+r L_K\boldsymbol{v}) \right\} \right)d\mu_{\cU} (\boldsymbol{v}).
\end{equation}
A typical assumption is that $g(u,\bar{\boldsymbol{\zeta}})<0$, i.e., the mean is inside the set whose probability we compute. This is naturally satisfied if $\varphi(u)\geq \frac{1}{2}$ and there exists a Slater point $\boldsymbol{z}\in\bbR^n$, s.t.\ $g(u,\boldsymbol{z})<0$. Then, the probabilistic function can be further simplified to
\begin{equation}\label{Eq:probfuncsimp2}
    \varphi(u)=\int_{\boldsymbol{v}\in\bbS^{K-1}}F_{\chi_K}\left( \rho(u,\boldsymbol{v}) \right)d\mu_{\cU} (\boldsymbol{v}),
\end{equation}
where $F_{\chi_K}$ is the cumulative distribution function of $\chi_K$, and
\begin{equation*}
    \rho(u,\boldsymbol{v}):=\sup_{r\geq 0}{\left\{ g(u,\bar{\boldsymbol{\zeta}}+r L_K\boldsymbol{v})\leq 0 \right\}}
\end{equation*}
is the length of the ray originating from $\bar{\boldsymbol\zeta}$ remaining within the feasible set, where we use the convention $F_{\chi_K}(\infty)=1$. Thus, we can approximate~\eqref{Eq:probfuncsimp2} by drawing samples $\{\boldsymbol{v}_i\}_{i=1}^N$ from $\cU(\bbS^{K-1})$ to obtain
\begin{equation}\label{eq:phiN(u)}
    \varphi(u)\approx \tilde{\varphi}_N(u):=\frac{1}{N}\sum_{i=1}^N F_{\chi_K}\left( \rho(u,\boldsymbol{v}_i) \right).
\end{equation}

Compared to standard Monte Carlo sampling to approximate \eqref{Eq:probfunc}, the SRD offers several advantages \cite{HeStWe-SIAM25, AcHe-SIAM14, AcHe-SIAM17, AcPe-AMO22}: (1) it provides derivatives of $\tilde\varphi_N(u)$ or $\varphi(u)$; (2) it allows the use of quasi-Monte Carlo sequences to sample the unit sphere $\bbS^{K-1}$, which typically yields faster convergence; and (3) it has a provably lower variance in important cases. In \cref{sec:hiSRD} we show how the SRD can be generalized to infinite stochastic dimensions by combining SRD with standard Monte Carlo sampling, resulting in an unbiased estimator. This requires a generalization of \eqref{Eq:probfuncsimp2}, since we cannot in general assume that rays originate from within the feasible set. In addition, we have to adjust the gradient expressions and revisit the differentiability arguments. To eliminate the bias introduced by truncation, another approach is simply increasing $K$. Beyond the potential numerical stability issues, this introduces a trade-off regarding variance reduction, which we will analyze next before introducing the proposed method.

\subsection{Degeneration of variance reduction}
The variance of the SRD estimator is always bounded above by that of standard Monte Carlo sampling~\cite{AcHe-SIAM14}, but is known to be reduced in important cases \cite{HeStWe-SIAM25}. However, numerical experiments indicate that as the dimension $K$ increases, this variance reduction weakens, and the SRD variance converges to that of standard Monte Carlo.
In this section, we analytically study this behavior and determine the rate at which the advantage of SRD vanishes. For simplicity, we fix some $u\in U_{\text{ad}}$ throughout the section and remove the dependence of $\rho$ on it.

To isolate the effect of high dimensionality, we consider the case where the random variable is contained within an $r$-dimensional subspace, but we apply an SRD of dimension $K$, where $r \ll K$. Specifically, we assume that only the first $r$ columns of $L_K$ are nonzero, that is, $L_K=[L_r, 0]$. In what follows, we study the finite-dimensional SRD variance as $K\to\infty$.

For any $\boldsymbol{z}\in\bbR^K$, let $\boldsymbol{z}^{(r)}\in\bbR^r$ denote its first $r$ components. Since the random variable $\boldsymbol{\nu}_K\sim\cU(\bbS^{K-1})$ can be expressed as $\boldsymbol{\nu}_K = \frac{\boldsymbol{\zeta}_K}{\norm{\boldsymbol{\zeta}_K}_2}$ with $\boldsymbol{\zeta}_K\sim\cN(0,I_K)$, we have
\[
\|{\boldsymbol{\nu}_K^{(r)}}\|_2^2=\frac{\norm{\boldsymbol{\zeta}_K^{(r)}}_2^2}{\norm{\boldsymbol{\zeta}_K}_2^2}=\frac{\norm{\boldsymbol{\zeta}_K^{(r)}}_2^2}{\norm{\boldsymbol{\zeta}_K^{(r)}}_2^2+\bigl( \norm{\boldsymbol{\zeta}_K}_2^2-\norm{\boldsymbol{\zeta}_K^{(r)}}_2^2 \bigr)}.
\]
Since $\norm{\boldsymbol{\zeta}_K^{(r)}}_2^2$ and $\norm{\boldsymbol{\zeta}_K}_2^2-\norm{\boldsymbol{\zeta}_K^{(r)}}_2^2$ are independently distributed as $\chi_r^2$ and $\chi_{K-r}^2$, respectively, $\norm{\boldsymbol{\nu}_K^{(r)}}_2^2$ follows a $\mathrm{Beta} \left( \frac{r}{2},\frac{K-r}{2} \right)$ distribution. We can further write $\boldsymbol{\nu}_K^{(r)}$ as
\begin{equation*}
    \boldsymbol{\nu}_K^{(r)}=\sqrt{\tau_{r,K}}\boldsymbol{\nu}_r,
\end{equation*}
where $\tau_{r,K}:=\norm{\boldsymbol{\nu}_K^{(r)}}_2^2\sim \mathrm{Beta} \left( \frac{r}{2},\frac{K-r}{2} \right)$ and $\boldsymbol{\nu}_r:={\boldsymbol{\nu}_K^{(r)}}/{\norm{\boldsymbol{\nu}_K^{(r)}}_2} \sim \cU(\bbS^{r-1})$.

The following lemma shows that by projecting samples from a high-dimensional sphere onto a low-dimensional space and integrating in one direction, we obtain the same result as obtained from the corresponding sample on the low-dimensional sphere. This relates SRD across different spherical dimensions.

\begin{lemma}\label{LEMMA:mean_K}
    For any $r\leq K$ and $c\geq 0$, it holds that
    \begin{equation}\label{Eq:mean_tau}
        \bbE_{\tau_{r,K}} \left[ F_{\chi_K}\left( \frac{1}{\sqrt{\tau_{r,K}}}c \right) \right] =F_{\chi_r}(c).
    \end{equation}
\end{lemma}
\begin{proof}
    Since $\sqrt{\tau_{r,K}}=\norm{\boldsymbol{\nu}_K^{(r)}}_2={\norm{\boldsymbol{\zeta}_K^{(r)}}_2}/{\norm{\boldsymbol{\zeta}_K}_2}$, and $\tau_{r,K}$ is independent of $\norm{\boldsymbol{\zeta}_K}_2$, we have
    \begin{equation*}
        \begin{aligned}
            F_{\chi_r}(c) &= \bbP\left( \norm{\boldsymbol{\zeta}_K^{(r)}}_2\leq c \right) 
            = \bbP\left( \norm{\boldsymbol{\zeta}_K}_2\leq \frac{c}{\sqrt{\tau_{r,K}}} \right) 
            = \bbE_{\tau_{r,K}}F_{\chi_K}\left( \frac{1}{\sqrt{\tau_{r,K}}}c \right),
        \end{aligned}
    \end{equation*}
    which proves the conclusion.
\end{proof}

Since $L_K\boldsymbol{\nu}_K=L_r\boldsymbol{\nu}_K^{(r)}$, we have
$    \rho_K(\boldsymbol{\nu}_K)=\rho_r(\boldsymbol{\nu}_K^{(r)})=\rho_r(\boldsymbol{\nu}_r)/{\sqrt{\tau_{r,K}}}$.
Using~\eqref{Eq:mean_tau}, we can write the variance of the $K$-dimensional SRD as
\begin{equation}\label{Eq:var_SRD_K}
    \begin{aligned}
        V_K &:= \mathrm{Var}_{\boldsymbol{\nu}_K} \left( F_{\chi_K}(\rho_K(\boldsymbol{\nu}_K)) \right)
        = \bbE_{\boldsymbol{\nu}_K} \left[\left( F_{\chi_K}(\rho_K(\boldsymbol{\nu}_K))-\varphi \right)^2\right] \\
        &= \bbE_{\boldsymbol{\nu}_r,\tau_{r,K}} \bigg[\Big( F_{\chi_K}\Big( \frac{1}{\sqrt{\tau_{r,K}}}\rho_r(\boldsymbol{\nu}_r) \Big)-\varphi \Big)^2\bigg] \\
        &= \bbE_{\boldsymbol{\nu}_r,\tau_{r,K}} \bigg[ \Big( F_{\chi_K}\Big( \frac{1}{\sqrt{\tau_{r,K}}}\rho_r(\boldsymbol{\nu}_r) \Big)-F_{\chi_r}(\rho_r(\boldsymbol{\nu}_r)) \Big)^2 + (F_{\chi_r}(\rho_r(\boldsymbol{\nu}_r))-\varphi)^2 \bigg].
    \end{aligned}
\end{equation}

Conversely, the variance of the $r$-dimensional SRD is given by
\begin{equation*}
    V_r := \mathrm{Var}_{\boldsymbol{\nu}_r} ( F_{\chi_r}(\rho_r(\boldsymbol{\nu}_r)) ) = \bbE_{\boldsymbol{\nu}_r} \big[ (F_{\chi_r}(\rho_r(\boldsymbol{\nu}_r))-\varphi)^2 \big],
\end{equation*}
which coincides with the second term in~\eqref{Eq:var_SRD_K}. Thus, while including redundant dimensions in SRD still yields the correct estimator, it introduces additional variance given by the first term in~\eqref{Eq:var_SRD_K}.
Next, we quantify this additional variance for large $K$.
First, the variance of the standard Monte Carlo estimation is
\begin{equation}\label{Eq:var_MC}
    \begin{aligned}
        V_{\mathrm{MC}} &:= \mathrm{Var}_{\boldsymbol{\nu}_r} \left( \cI_{\left\{ \tau_r\leq \rho_r(\boldsymbol{\nu}_r) \right\}} \right) = \bbE_{\boldsymbol{\nu}_r,\tau_r} \left[\left( \cI_{\left\{ \tau_r\leq \rho_r(\boldsymbol{\nu}_r) \right\}}-\varphi \right)^2\right] \\
        &= \bbE_{\boldsymbol{\nu}_r,\tau_r} \left[ \left( \cI_{\left\{ \tau_r\leq \rho_r(\boldsymbol{\nu}_r) \right\}}-F_{\chi_r}(\rho_r(\boldsymbol{\nu}_r)) \right)^2 + (F_{\chi_r}(\rho_r(\boldsymbol{\nu}_r))-\varphi)^2 \right] \\
        &= \bbE_{\boldsymbol{\nu}_r,\tau_r} \left[ F_{\chi_r}(\rho_r(\boldsymbol{\nu}_r))\left( 1-F_{\chi_r}(\rho_r(\boldsymbol{\nu}_r)) \right) + (F_{\chi_r}(\rho_r(\boldsymbol{\nu}_r))-\varphi)^2 \right].
    \end{aligned}
\end{equation}
Denote
\begin{equation}\label{Eq:Wk}
    \begin{aligned}
        W_K(c)&:=\bbE_{\tau_{r,K}} \bigg[ \Big( F_{\chi_K}\Big( \frac{1}{\sqrt{\tau_{r,K}}}c \Big)-F_{\chi_r}(c) \big)^2 \bigg], \\
        W_{\mathrm{MC}}(c)&:=F_{\chi_r}(c)\left( 1-F_{\chi_r}(c) \right),
    \end{aligned}
\end{equation}
then by~\eqref{Eq:var_SRD_K} and~\eqref{Eq:var_MC}, the difference between $V_K$ and $V_{\mathrm{MC}}$ is
\begin{equation}\label{Eq:V_diff}
    V_{\mathrm{MC}}-V_K=\bbE_{\boldsymbol{\nu}_r} \left[ W_{\mathrm{MC}}(\rho_r(\boldsymbol{\nu}_r))-W_K(\rho_r(\boldsymbol{\nu}_r)) \right],
\end{equation}
which means that $W_{\mathrm{MC}}(\rho_r(\boldsymbol{\nu}_r))-W_K(\rho_r(\boldsymbol{\nu}_r))$ is the variance difference in the direction $\boldsymbol{\nu}_r$.

The following theorem, whose proof is detailed in \cref{app:1}, shows that $V_{\mathrm{MC}}-V_K$ converges to $0$ as $K\to\infty$, with a rate of $\cO \big( K^{-\frac{1}{2}} \big)$.

\begin{theorem}\label{thm:estimation-degeneration}
    There exists a uniform constant $C>0$, such that for any $c>0$, $W_{\mathrm{MC}}(c)-W_K(c)\leq \frac{C}{\sqrt{K}}$. Consequently, $V_{\mathrm{MC}}-V_K\leq \frac{C}{\sqrt{K}}$. 
\end{theorem}

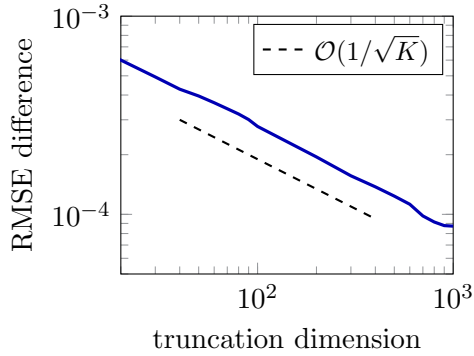
\begin{figure}[tbp]
    \centering
    \begin{tikzpicture}
        \begin{loglogaxis}[
            xlabel={truncation dimension},
            ylabel={RMSE difference},
            xmin=20, xmax=1000,
            ymin=5e-5, ymax=1e-3,
            width=6cm, height=5cm,
            legend pos=north east,
            legend cell align={left},
            legend style={font=\small},
            grid style=dotted,
        ]
        \path[name path=axis] (axis cs:1,1e-4) -- (axis cs:1000,1e-4);
        \addplot[color=black!30!blue, mark=none, very thick, forget plot]
            table[x=k, y expr=(\thisrow{MC}-\thisrow{hiSRD_MC})] {data/hiSRD.txt};
        \addplot[dashed,black,thick] coordinates {
            (40,3e-4)
            (400,3e-4/sqrt{10})
        };
        \addlegendentry{$\cO(1/\sqrt{K})$}
        \end{loglogaxis} 
    \end{tikzpicture}
    \caption{RMSE difference between standard MC and SRD ($y$-axis) versus approximation dimension ($x$-axis). The dashed line indicates the slope theoretically expected from \cref{thm:estimation-degeneration}. Results are for Gaussian process example detailed in \cref{SEC:NumGPR}, but with 1000 repeats.}\label{fig:var_diff}
\end{figure}

The theorem provides an upper bound on the variance differences, and the question arises if one observes the $1/\sqrt{K}$ rate in practice. We therefore use the same data as in \cref{fig:intro} and, in \cref{fig:var_diff}, plot the difference in the variance-dominated regime against $K$.
Note that the computed RMSE difference is an estimation of $(\sqrt{V_{\mathrm{MC}}}-\sqrt{V_K})/\sqrt{N}$, which has the same convergence order as $V_{\mathrm{MC}}-V_K$. We thus observe that \cref{thm:estimation-degeneration} accurately describes the variance degeneration as a function of $K$.

\section{Hybrid infinite-dimensional spherical-radial decomposition}\label{sec:hiSRD}
We next motivate and introduce the proposed \hiSRD\ method and analyze the components contributing to its variance. A central focus of the section is to establish the differentiability of probability function estimates obtained with \hiSRD, which requires extension of existing techniques. We also discuss the computational complexity of the method.

\subsection{Definition of the method}
Employing a truncated stochastic expansion, the SRD described in \cref{subsec:finite-dim-SRD} provides an efficient and differentiable method to approximate chance constraints. However, the choice of the truncation dimension $K$, requires careful consideration. If $K$ is too small, the resulting probability estimation can be heavily biased. If $K$ is too large, the variance reduction compared to standard MC becomes marginal, as demonstrated in \cref{subsec:variance-analysis}. To address this, we now introduce \hiSRD, which eliminates this bias and renders the accuracy largely independent of the choice of $K$. This is achieved by adding a standard Monte Carlo correction term that accounts for the truncation remainder.

\begin{figure}[bt]
    \centering
    \includegraphics[width=0.38\linewidth]{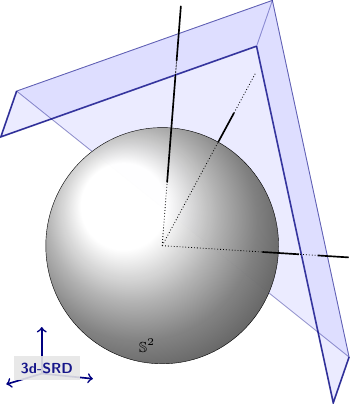}\hspace{.06\textwidth}
    \includegraphics[width=0.38\linewidth]{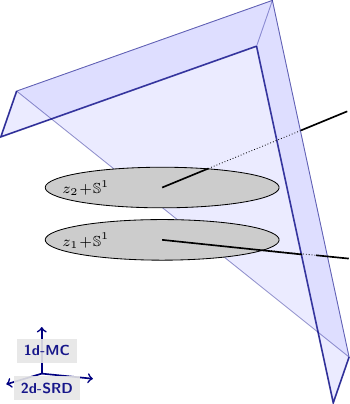}
   \caption{Illustration of standard SRD (left) and hybrid SRD (right) to estimate the probability for the three-dimensional set given by the intersection of half spaces defined by hyperplanes (light blue). The standard SRD integrates the three-dimensional $\chi$-density along 3D rays (black lines in left figure). The hybrid method uses Monte Carlo in the $z$-direction and SRD in the $x,y$-direction, integrating the two-dimensional $\chi$-density along rays (black lines in right figure). }
    \label{fig:hiSRD-illustration}
\end{figure}

In infinite-dimensional spaces, the standard geometric definition of elliptical distributions breaks down since there is no uniform probability measure on the unit sphere. Instead, a random variable $\xi$ taking values in a separable Hilbert space $\mathscr{H}$ is said to follow an elliptical distribution if it can be represented as a scale mixture of a Gaussian measure \cite{BoBaTy-JMA14}. Specifically, this means that $\xi$ can be written as $\xi = \bar{\xi} + \sigma\zeta$,
where $\bar{\xi} \in \mathscr{H}$ denotes the mean of $\xi$, $\zeta$ is a centered Gaussian random variable in $\mathscr{H}$ with a trace-class covariance operator, and $\sigma$ is a non-negative scalar random variable independent of $\zeta$.

Since every infinite-dimensional elliptical random variable is inherently built upon a Gaussian base $\zeta$, applying SRD to a general elliptical distribution requires performing the decomposition on the underlying Gaussian measure and composing the radial distribution with $\sigma$. Consequently, for simplicity of the presentation, we focus on the case where $\xi$ is a Gaussian random field, i.e., $\sigma$ is a constant. The generalization to other elliptical distributions, such as the infinite-dimensional Student's $t$-distribution, only requires modifying the radial component. Specifically, the radial distribution $\chi_K$ below must be replaced by the distribution of the product $\sigma \cdot \chi_K$.

Suppose $\xi$ is an infinite-dimensional Gaussian random variable with a KL expansion of the form~\eqref{eq:inf-xi}. We can truncate the expansion and decompose the centered random variable $\xi-\bar{\xi}$ into $\xi-\bar{\xi}=\xi_K+\xi_R$, where
\begin{equation*}
    \xi_K:=\sum_{k=1}^K \xi_k\phi_k,\quad \xi_R:=\sum_{k=K+1}^\infty \xi_k\phi_k.
\end{equation*}
We decompose the space correspondingly by $\mathscr{H}=\mathscr{H}_K\oplus \mathscr{H}_R$, where $\mathscr{H}_K:=\operatorname{span}{\{\phi_k:k\leq K\}}$, and $\mathscr{H}_R:=\operatorname{span}{\{\phi_k:k\geq K+1\}}$. We perform SRD on $\xi_K$, and keep the integration of the remainder $\xi_R$:
\begin{equation}\label{Eq:iint_mu}
 \varphi(u)=\int_{\mathscr{H}_R}\int_{\boldsymbol{v}\in\bbS^{K-1}}\mu_{\chi_K}\left(\left\{r\geq 0: g(u,\bar{\xi}+z+r L_K\boldsymbol{v})\leq 0\right\}\right)d\mu_{\cU} (\boldsymbol{v})d\mu_{\xi_R}(z).
\end{equation}
Here, $z$ represents a realization of the remainder $\xi_R$, and $L_K:\bbR^K\to\mathscr{H}_K$ is defined by $L_K\boldsymbol{x}=\sum_{k=1}^K x_k \phi_k$ for $\boldsymbol{x}\in\bbR^K$. The difference between the SRD in finite dimensions and the definition in \eqref{Eq:iint_mu} is depicted schematically in \cref{fig:hiSRD-illustration}. Note that even if $g(u,\bar\xi)<0$, we generally cannot conclude that $g(u,\bar{\xi}+z)<0$. Compared to the SRD in finite dimensions, we therefore need to modify integration along rays $\boldsymbol v$ and define
\begin{equation*}
    \begin{gathered}
        \rho_{\mathrm{in}}(u,\boldsymbol{v},z):=\inf{\{r\geq 0:g(u,\bar{\xi}+z+r L_K\boldsymbol{v})\leq 0\}}, \\
        \rho_{\mathrm{out}}(u,\boldsymbol{v},z):=\sup{\{r\geq 0:g(u,\bar{\xi}+z+r L_K\boldsymbol{v})\leq 0\}}.
    \end{gathered}
\end{equation*}
Here, when considering extrema over subsets of $[0,+\infty)$, we use the conventions $\inf{\varnothing}=+\infty$ and $\sup{\varnothing}=0$. Assuming that $g(u,\cdot)$ is convex for every $u$, the one-dimensional set $\{r\geq 0: g(u,\bar{\xi}+z+r L_K\boldsymbol{v})\leq 0\}$ is the interval $[\rho_{\mathrm{in}}(u,\boldsymbol{v},z),\rho_{\mathrm{out}}(u,\boldsymbol{v},z)]$, and thus
\begin{equation*}
    \mu_{\chi_K}\big(\{r\geq 0: g(u,\bar{\xi}+z+r L_K\boldsymbol{v})\leq 0\}\big) = \max{\{0,F_{\!\chi_K}(\rho_{\mathrm{out}}(u,\boldsymbol{v},z))\!-\!F_{\!\chi_K}(\rho_{\mathrm{in}}(u,\boldsymbol{v},z))\}}.
\end{equation*}
Consequently, we can rewrite the probabilistic function~\eqref{Eq:iint_mu} as
\begin{equation*}
    \varphi(u)=\int_{\mathscr{H}}\int_{\boldsymbol{v}\in\bbS^{K-1}}\max{\{0, F_{\chi_K}(\rho_{\mathrm{out}}(u,\boldsymbol{v},z))-F_{\chi_K}(\rho_{\mathrm{in}}(u,\boldsymbol{v},z))\}}d\mu_{\cU} (\boldsymbol{v})d\mu_{\xi_R}(z).
\end{equation*}
By independently sampling $\{\boldsymbol{v}_i\}_{i=1}^N$ from $\cU(\bbS^{K-1})$ and $\{z_i\}_{i=1}^N$ from $\xi_R$, we obtain an unbiased approximation $\tilde{\varphi}_N$ of $\varphi$:
\begin{equation}\label{Eq:phi_u_dis}
    \varphi(u)\approx \tilde{\varphi}_N(u):=\frac{1}{N}\sum_{i=1}^N \left[ \max{\{0, F_{\chi_K}(\rho_{\mathrm{out}}(u,\boldsymbol{v}_i,z_i))-F_{\chi_K}(\rho_{\mathrm{in}}(u,\boldsymbol{v}_i,z_i))\}} \right].
\end{equation}

\subsection{Variance analysis}\label{subsec:variance-analysis}

The variance of the \hiSRD\ estimator can be decomposed into two components: the expected variance of the SRD estimator conditioned on the remainder, and the variance introduced by the remainder $\xi_R$ itself. To formalize this, for any fixed remainder realization $z\in\mathscr{H}_R$ and direction $\boldsymbol{v}\in\bbS^{K-1}$, we define the ray probability as
\begin{equation*}
    \tilde{\varphi}(u;\boldsymbol{v},z) := \mu_{\chi_K}\left(\left\{r\geq 0: g(u,\bar{\xi}+z+r L_K\boldsymbol{v})\leq 0\right\}\right).
\end{equation*}
Integrating this over the sphere yields the exact conditional probability $\bbP(g(u,\xi)\leq 0 \mid \xi_R=z)$, which we denote by
\begin{equation*}
    \tilde{\varphi}_S(u;z) := \int_{\boldsymbol{v}\in\bbS^{K-1}}\tilde{\varphi}(u;\boldsymbol{v},z) d\mu_{\cU} (\boldsymbol{v}).
\end{equation*}
Consequently, the following expectation relations hold:
\begin{equation*}
    \bbE_{\boldsymbol{\nu}}[\tilde{\varphi}(u;\boldsymbol{\nu},\xi_R) \mid \xi_R] = \tilde{\varphi}_S(u;\xi_R),\quad \bbE_{\xi_R}[\tilde{\varphi}_S(u;\xi_R)] = \varphi(u).
\end{equation*}
Using the law of total variance, the variance of the estimator $\tilde{\varphi}(u;\boldsymbol{\nu},\xi_R)$ can be split as
\begin{equation*}
    \begin{aligned}
        V_{\text{total}}&:= \mathrm{Var}_{\boldsymbol{\nu},\xi_R}(\tilde{\varphi}(u;\boldsymbol{\nu},\xi_R)) =  V_{\text{SRD}} + V_{\text{rem}},
    \end{aligned}
\end{equation*}
    where
\begin{equation}\label{eq:VSRD}
        V_{\text{SRD}} := \bbE_{\xi_R}[\mathrm{Var}_{\boldsymbol{\nu}}(\tilde{\varphi}(u;\boldsymbol{\nu},\xi_R) \mid \xi_R)], \quad V_{\text{rem}}:= \mathrm{Var}_{\xi_R}(\tilde{\varphi}_S(u;\xi_R)).
\end{equation}
Here, 
$V_{\text{SRD}}$
represents the expected conditional variance of the SRD estimator, where the expectation is taken over the remainder $\xi_R$, and $V_{\text{rem}}$
captures the variance of $\tilde{\varphi}_S(u;\xi_R)$ with respect to the remainder, and is independent of $\boldsymbol{\nu}$.

With moderately large truncation dimension $K$, one can expect the variance of the remainder $\xi_R$ to be small, so we anticipate that $V_{\text{rem}}$ is well bounded by the variance of the remainder itself and is relatively small. Since SRD yields substantial variance reduction in practice, $V_{\text{SRD}}$ generally remains smaller than the variance of standard Monte Carlo. The behavior of these two variance components is further examined numerically in~\cref{subsec:numerical-variance}.

\subsection{Differentiability of probability functions}
\begin{figure}[tb]
    \centering
    \includegraphics[width=0.42\linewidth]{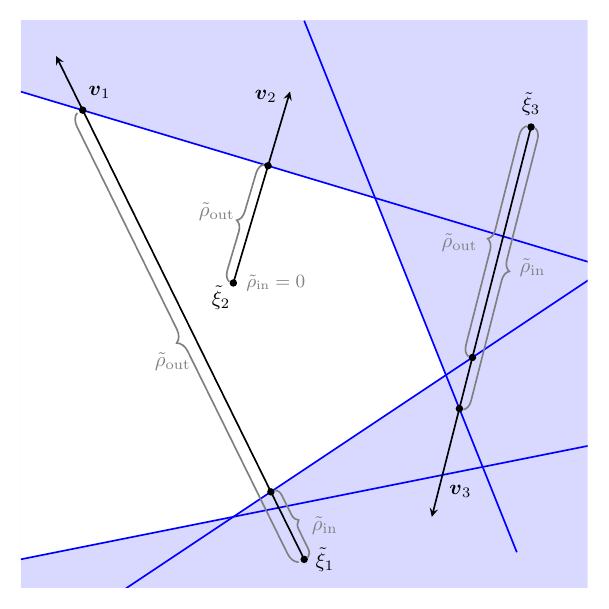}
    \caption{Geometric illustration of SRD rays $\tilde{\xi}_i+r\boldsymbol{v}_i$ ($i=1,2,3$), which start at $\tilde{\xi}_i$. The affine constraints defining the feasible set (white area) are shown as half spaces (in blue). The rays demonstrate three different intersection scenarios: a ray entering and exiting the feasible set from the outside, with $0<\tilde{\rho}_{\mathrm{in}}<\tilde{\rho}_{\mathrm{out}}$ ($i=1$), a ray originating inside the feasible set, with $0=\tilde{\rho}_{\mathrm{in}}<\tilde{\rho}_{\mathrm{out}}$ ($i=2$), and a ray that completely misses the feasible set, illustrating the case in which $\tilde{\rho}_{\mathrm{out}} < \tilde{\rho}_{\mathrm{in}}$ ($i=3$).}
    \label{Fig:SRDrays}
\end{figure}

The differentiability of the finite-dimen\-sional SRD estimator has been studied in different settings. Prior works have analyzed either the exact probabilistic function $\varphi(\cdot)$ defined in \eqref{eq:phi(u)}, based on analytical integration on the unit sphere \cite{AcPe-AMO22,AcHe-SIAM17,AcHe-SIAM14}, or the discrete approximation $\varphi_N(\cdot)$ defined in \eqref{eq:phiN(u)}, which is based on Monte Carlo sampling from the uniform distribution on the unit sphere \cite{AcHeZi-AMO24,HeStWe-SIAM25}. Here we follow the latter approach. We build on the techniques leading to \cite[Cor.\ 3.2]{AcHeZi-AMO24}, where it is shown that if $g$ is a continuously differentiable function that is convex in the second variable and satisfies certain growth conditions, and $g(u,\bar{\xi})<0$, then $F_{\chi}(\rho(u,\boldsymbol{v}))$ is differentiable. These results must be generalized to show the differentiability of \hiSRD, as the shifted means $\bar{\xi}+z$ may not satisfy $g(u,\bar{\xi}+z)<0$ for every sample $z$.

Assume that the random constraint function $g$ has the specific form
\begin{equation*}
    g:=\max_{j=1,\dots,M}g_j,
\end{equation*}
where for any $u\in U$, $g_j(u,\cdot)$ is affine. Subsequently, we fix some $u\in U_{\text{ad}}$, and assume that
\begin{equation}\label{Eq:neg_gj}
    g_j(u,\bar{\xi})<0,\quad j=1,\dots,M.
\end{equation}

To compute derivatives of $\tilde{\varphi}_N(u)$ defined in \eqref{Eq:phi_u_dis}, we first derive an equivalent form of $\tilde{\varphi}_N$. For that purpose, we partition the indices into two groups, according to whether $g_j(u,\bar{\xi}+z)$ satisfies the constraint or not:
\begin{gather*}
    J_{\mathrm{i}}:=\{j\in\{1,\dots,M\}:g_j(u,\bar{\xi}+z)< 0\}, \\
    J_{\mathrm{o}}:=\{j\in\{1,\dots,M\}:g_j(u,\bar{\xi}+z)> 0\}.
\end{gather*}
We then define
\begin{gather*}
    \begin{aligned}
        \rho_{j,\mathrm{in}}(u,\boldsymbol{v},z) &:= \inf{\{r\geq 0:g_j(u,\bar{\xi}+z+r L_K\boldsymbol{v})\leq 0\}} \\
        &= \sup{\{r\geq 0:-g_j(u,\bar{\xi}+z+r L_K\boldsymbol{v})\leq 0\}},
    \end{aligned} \\
    \rho_{j,\mathrm{out}}(u,\boldsymbol{v},z):=\sup{\{r\geq 0:g_j(u,\bar{\xi}+z+r L_K\boldsymbol{v})\leq 0\}},
\end{gather*}
and
\begin{equation}\label{Eq:rhoModified}
    \tilde{\rho}_{\mathrm{in}}(u,\boldsymbol{v},z):=\max_{j\in J_{\mathrm{o}}}{\rho_{j,\mathrm{in}}(u,\boldsymbol{v},z)},\quad \tilde{\rho}_{\mathrm{out}}(u,\boldsymbol{v},z):=\min_{j\in J_{\mathrm{i}}}{\rho_{j,\mathrm{out}}(u,\boldsymbol{v},z)}.
\end{equation}
The geometric meanings of $\tilde{\rho}_{\mathrm{in}}$ and $\tilde{\rho}_{\mathrm{out}}$ are illustrated in~\cref{Fig:SRDrays}.

The modified distance definitions \eqref{Eq:rhoModified} are not only defined for theoretical derivation but are also used in numerical calculations. The typical finite-dimensional setting of SRD summarized in \cref{subsec:finite-dim-SRD} only considers centers $\bar{\xi}$ where $g(u,\bar{\xi})<0$. In our generalized setting, the shifted means $\bar{\xi}+z$ may fall outside the feasible set for some realizations of the remainder $z$. When this occurs, evaluating $\rho_{\mathrm{in}}$, $\rho_{\mathrm{out}}$ in~\eqref{Eq:phi_u_dis} is problematic, since $\rho_{\mathrm{out}}=\min_{j\in J}{\rho_{j,\mathrm{out}}(u,\boldsymbol{v},z)}$ may not hold. However, we can still use the distances defined in~\eqref{Eq:rhoModified} to construct an alternative, computationally efficient form of $\tilde{\varphi}_N(u)$, as shown in the following lemma.

\begin{lemma}
    In the setting and with the above definitions, it holds that:
    \begin{equation}\label{Eq:phi_u_dis2}
        \tilde{\varphi}_N(u)=\frac{1}{N}\sum_{i=1}^N \left[ \max{\{0, F_{\chi_K}(\tilde{\rho}_{\mathrm{out}}(u,\boldsymbol{v}_i,z_i))-F_{\chi_K}(\tilde{\rho}_{\mathrm{in}}(u,\boldsymbol{v}_i,z_i))\}} \right].
    \end{equation}
\end{lemma}

\begin{proof}
    For each constraint $j \in \{1, \dots, M\}$, we have
    \begin{equation*}
        I_j := \{r \ge 0 : g_j(u,\bar{\xi}+z+r L_K\boldsymbol{v}) \le 0\} = [\rho_{j,\mathrm{in}}, \rho_{j,\mathrm{out}}].
    \end{equation*}
    Specifically, for $j \in J_{\mathrm{i}}$, the center $\bar{\xi}+z$ is inside the feasible half-space. Therefore, $\rho_{j,\mathrm{in}} = 0$ and $I_j = [0, \rho_{j,\mathrm{out}}]$. For $j \in J_{\mathrm{o}}$, the center is outside, and the ray never leaves the feasible half-space after intersecting it. Therefore, $\rho_{j,\mathrm{out}} = \infty$ and $I_j = [\rho_{j,\mathrm{in}}, \infty)$.
    
    The overall feasible set along the ray is the intersection of these intervals, given by the (possibly empty) interval:
    \begin{equation*}
        \begin{aligned}
            \left\{r\geq 0: g(u,\bar{\xi}+z+r L_K\boldsymbol{v})\leq 0\right\} &= \bigcap_{j=1}^M I_j = 
            [\tilde{\rho}_{\mathrm{in}}(u,\boldsymbol{v},z), \tilde{\rho}_{\mathrm{out}}(u,\boldsymbol{v},z)].
        \end{aligned}
    \end{equation*}
  If this intersection is empty, i.e., $\tilde{\rho}_{\mathrm{in}} > \tilde{\rho}_{\mathrm{out}}$, then the operator $\max\{0,\cdot\}$ in~\eqref{Eq:phi_u_dis2} correctly sets the negative CDF difference to $0$. Consequently, applying the same reasoning used to obtain~\eqref{Eq:phi_u_dis}, $\tilde{\varphi}_N(u)$ can be equivalently expressed as in~\eqref{Eq:phi_u_dis2}.
\end{proof}

We now show the differentiability of $\tilde{\varphi}_N(u)$ and derive explicit gradient expressions based on the form~\eqref{Eq:phi_u_dis2}. Specifically, we establish that $\tilde{\varphi}_N(u)$ is differentiable for a specific sample pair $\{\boldsymbol{v}_i\}_{i=1}^N$, $\{z_i\}_{i=1}^N$ under the following assumptions.

\begin{assumption}\label{Ass:sample}
    For fixed $u\in U$ and sample pair $\{\boldsymbol{v}_i\}_{i=1}^N$, $\{z_i\}_{i=1}^N$, assume that for all $i\in\{1,\dots,N\}$:
    \begin{enumerate}[leftmargin=*]
        \item \label{Ass:index} $J_{\mathrm{i}}\cup J_{\mathrm{o}}=\{1,\dots,M\}$;
        \item \label{Ass:inout} $\tilde{\rho}_{\mathrm{in}}(u,\boldsymbol{v}_i,z_i) \neq \tilde{\rho}_{\mathrm{out}}(u,\boldsymbol{v}_i,z_i)$ if $\tilde{\rho}_{\mathrm{in}}(u,\boldsymbol{v}_i,z_i) < \infty$ or $\tilde{\rho}_{\mathrm{out}}(u,\boldsymbol{v}_i,z_i) < \infty$;
        \item \label{Ass:unique} It holds that
        \begin{align*}
            &\#\left\{ j\in J_{\mathrm{o}} : \tilde{\rho}_{\mathrm{in}}(u,\boldsymbol{v}_i,z_i)=\rho_{j,\mathrm{in}}(u,\boldsymbol{v}_i,z_i) \right\}=1 \quad &&\text{if $0<\tilde{\rho}_{\mathrm{in}}(u,\boldsymbol{v}_i,z_i)<\infty$,}\\
            &\#\left\{ j\in J_{\mathrm{i}} : \tilde{\rho}_{\mathrm{out}}(u,\boldsymbol{v}_i,z_i)=\rho_{j,\mathrm{out}}(u,\boldsymbol{v}_i,z_i) \right\}=1 \quad &&\text{if $\tilde{\rho}_{\mathrm{out}}(u,\boldsymbol{v}_i,z_i)<\infty$}.
        \end{align*}
    \end{enumerate}
\end{assumption}

Under~\cref{Ass:sample}, we use $\tilde{\rho}_{\mathrm{in}}^{(i)}$ and $\tilde{\rho}_{\mathrm{out}}^{(i)}$ to denote the corresponding unique indices in~\eqref{Eq:rhoModified} for the $i$-th sample, i.e.,
\begin{equation*}
    \tilde{\rho}_{\mathrm{in}}^{(i)}:=\tilde{\rho}_{\mathrm{in}}(u,\boldsymbol{v}_i,z_i),\quad \tilde{\rho}_{\mathrm{out}}^{(i)}:=\tilde{\rho}_{\mathrm{out}}(u,\boldsymbol{v}_i,z_i),
\end{equation*}
and use $j_{1,i}$, $j_{2,i}$ to denote the corresponding unique indices 
for $\tilde{\rho}_{\mathrm{in}}^{(i)}$ and $\tilde{\rho}_{\mathrm{out}}^{(i)}$. We denote $\tilde{\xi}_i := \bar{\xi} + z_i$. We use $\nabla_z g_j$ to denote the gradient of $g_j$ with respect to its second argument. We also use the convention that $f_{\chi_K}(\infty) = 0$.

The next proposition generalizes existing differentiability results for finite-dimensional SRD, such as \cite[Prop.~3.4]{HeStWe-SIAM25}. Since we allow $g(u,\tilde{\xi}_i) \geq 0$, differentiating $\tilde{\varphi}_N(u)$ requires considering contributions from both the entry distance $\tilde{\rho}_{\mathrm{in}}^{(i)}$ and the exit distance $\tilde{\rho}_{\mathrm{out}}^{(i)}$.
Degenerate geometric cases 
which have origins lying exactly on the boundary or rays intersecting the feasible set at a single point, are ruled out by assumptions \ref{Ass:index} and \ref{Ass:inout}.

\begin{proposition}\label{Prop:differentiability}
    Under~\cref{Ass:sample}, $\tilde{\varphi}_N(u)$ is differentiable at $u$, and the derivative is given by
    \begin{equation}\label{Eq:GradDis}
        \begin{aligned}
            \nabla \tilde{\varphi}_N(u) &= \frac{1}{N}\sum_{i=1}^N \Bigg( -\frac{f_{\chi_K}(\tilde{\rho}_{\mathrm{out}}^{(i)})}{\braket{\nabla_{z}g_{j_{2,i}}(u,\tilde{\xi}_i+\tilde{\rho}_{\mathrm{out}}^{(i)} L_K\boldsymbol{v}_i), L_K\boldsymbol{v}_i}}\nabla_u g_{j_{2,i}}(u,\tilde{\xi}_i+\tilde{\rho}_{\mathrm{out}}^{(i)} L_K\boldsymbol{v}_i) \\
            &\quad  + \cI_{\{\tilde{\rho}_{\mathrm{in}}^{(i)} > 0\}} \frac{f_{\chi_K}(\tilde{\rho}_{\mathrm{in}}^{(i)})}{\braket{\nabla_{z}g_{j_{1,i}}(u,\tilde{\xi}_i+\tilde{\rho}_{\mathrm{in}}^{(i)} L_K\boldsymbol{v}_i), L_K\boldsymbol{v}_i}}\nabla_u g_{j_{1,i}}(u,\tilde{\xi}_i+\tilde{\rho}_{\mathrm{in}}^{(i)} L_K\boldsymbol{v}_i) \Bigg) H_i,
        \end{aligned}
    \end{equation}
    where $H_i = 1$ if $\tilde{\rho}_{\mathrm{out}}^{(i)} > \tilde{\rho}_{\mathrm{in}}^{(i)}$ and $0$ otherwise, and $\cI_{\{\tilde{\rho}_{\mathrm{in}}^{(i)} > 0\}}$ is the indicator function taking value $1$ when $\tilde{\rho}_{\mathrm{in}}^{(i)} > 0$ and $0$ otherwise.
\end{proposition}

\begin{proof}
    Assumption~\ref{Ass:index} implies that the index sets $J_{\mathrm{i}}$ and $J_{\mathrm{o}}$ do not change within some neighborhood of $u$. Then, by assumption~\ref{Ass:inout}, we only need to show that $F_{\chi_K}(\tilde{\rho}_{\mathrm{out}}^{(i)})$ and $F_{\chi_K}(\tilde{\rho}_{\mathrm{in}}^{(i)})$ are differentiable. For $j=1,\dots,M$, we define:
    \begin{equation*}
        \rho_j:=\left\{\begin{array}{ll}
            \rho_{j,\mathrm{in}}, & j\in J_{\mathrm{o}}, \\
            \rho_{j,\mathrm{out}}, & j\in J_{\mathrm{i}}.
        \end{array}\right.
    \end{equation*}
    Since both $g_j$ and $-g_j$ are affine, by~\cite[Cor.\ 3.2]{AcHeZi-AMO24}, $F_{\chi_K}(\rho_{j})=F_{\chi_K}(\rho_{j}(u,\boldsymbol{v},\boldsymbol{z}))$ is continuously differentiable at all $(u,\boldsymbol{v},\boldsymbol{z})\in \cV_j(u)\times \bbS^{K-1}\times \bbR^{n}$, where $\cV_j(u)$ is some neighborhood of $u$. The gradient is given by
    \begin{equation}\label{Eq:gradient_j}
        \nabla_u F_{\chi_K}(\rho_{j})=-\frac{f_{\chi_K}(\rho_{j})}{\braket{\nabla_z g_j(u,\tilde{\xi}+\rho_{j} L_K\boldsymbol{v}), L_K\boldsymbol{v}}}\nabla_u g_{j}(u,\tilde{\xi}+\rho_{j} L_K\boldsymbol{v}),
    \end{equation}
    where $\tilde{\xi}:=\bar{\xi}+z$.

    Now, first consider an index $i\in\{1,\dots,N\}$ where $\tilde{\rho}_{\mathrm{out}}^{(i)}<\infty$ and $0<\tilde{\rho}_{\mathrm{in}}^{(i)}<\infty$, which implies that both $J_{\mathrm{i}}$ and $J_{\mathrm{o}}$ are nonempty.
    Thus, by~\eqref{Eq:gradient_j},
    \begin{equation}\label{Eq:grad_rhoout}
        \nabla_u F_{\chi_K}(\tilde{\rho}_{\mathrm{out}}^{(i)}) = -\frac{f_{\chi_K}(\tilde{\rho}_{\mathrm{out}}^{(i)})}{\braket{\nabla_{z}g_{j_{2,i}}(u,\tilde{\xi}_i+\tilde{\rho}_{\mathrm{out}}^{(i)} L_K\boldsymbol{v}_i), L_K\boldsymbol{v}_i}}\nabla_u g_{j_{2,i}}(u,\tilde{\xi}_i+\tilde{\rho}_{\mathrm{out}}^{(i)} L_K\boldsymbol{v}_i),
    \end{equation}
    and analogously for $\nabla_u F_{\chi_K}(\tilde{\rho}_{\mathrm{in}}^{(i)})$.

    Next, consider an index $i$ where $\tilde{\rho}_{\mathrm{in}}^{(i)}=0$. This implies $J_{\mathrm{o}}=\varnothing$, since from the definition of $J_{\mathrm{o}}$, $\rho_{j,\mathrm{in}}>0$ for all $j\in J_{\mathrm{o}}$. Therefore, $J_{\mathrm{o}}=\varnothing$ in some neighborhood of $u$, and thus $\nabla_u F_{\chi_K}(\tilde{\rho}_{\mathrm{in}}^{(i)})=0$. The indicator $\cI_{\{\tilde{\rho}_{\mathrm{in}}^{(i)} > 0\}}$ in~\eqref{Eq:GradDis} explicitly zeroes out this term.

    Lastly, consider an index $i$ where $\tilde{\rho}_{\mathrm{out}}^{(i)}=\infty$ or $\tilde{\rho}_{\mathrm{in}}^{(i)}=\infty$. The case $\tilde{\rho}_{\mathrm{out}}^{(i)}=\infty$ and $J_{\mathrm{i}}=\varnothing$ is the same as the previous one. If $\tilde{\rho}_{\mathrm{out}}^{(i)}=\infty$ and $J_{\mathrm{i}}\neq\varnothing$, then by~\cite[Lem~3.3]{AcHeZi-AMO24}, $\nabla_u F_{\chi_K}(\tilde{\rho}_{\mathrm{out}}^{(i)})=0$. If $\tilde{\rho}_{\mathrm{in}}^{(i)}=\infty$, the same lemma still applies but in a slightly different way. In this case, \cite[Lem~3.3]{AcHeZi-AMO24} gives $\nabla_u F_{\chi_K}(\rho_{j_{1,i}}(u,\boldsymbol{v}_i,z_i))=0$. Also, we have $F_{\chi_K}(\rho_{j_{1,i}}(u,\boldsymbol{v}_i,z_i))=1$. Since for any $j=1,\dots,M$, $F_{\chi_K}(\rho_{j}(u,\boldsymbol{v}_i,z_i))\leq 1$, it still holds that $\nabla_u F_{\chi_K}(\tilde{\rho}_{\mathrm{in}}^{(i)})=0$.

    This shows that the gradient expression~\eqref{Eq:grad_rhoout} and the analogue for $\nabla_u F_{\chi_K}(\tilde{\rho}_{\mathrm{in}}^{(i)})$ are valid for any index $i$, so the proof is complete.
\end{proof}

Note that since $f_{\chi_1}(0) > 0$, the indicator function $\cI_{\{\tilde{\rho}_{\mathrm{in}}^{(i)} > 0\}}$ is necessary for the case $K=1$ to ensure that $\nabla_u F_{\chi_K}(\tilde{\rho}_{\mathrm{in}}^{(i)})$ is correctly zeroed out when $\tilde{\rho}_{\mathrm{in}}^{(i)} = 0$. For $K > 1$, the fact that $f_{\chi_K}(0) = 0$ naturally eliminates the term.

The following proposition shows that \cref{Ass:sample} is almost surely satisfied under the rank-2-constraint qualification. A related proposition was established in \cite[Lem.~4.3]{AcHe-SIAM17}, and we generalize the argument to incorporate the stochastic remainders and to accommodate shifted centers outside the feasible set.

\begin{proposition}
    If the following rank-2-constraint qualification (R2CQ)
    \begin{equation}\label{Eq:R2CQ}
        \begin{gathered}
            \forall z\in \mathscr{H}, i,j\in\{1,\dots,M\},i\neq j:
            g_i(u,z)=g_j(u,z)=0 \Longrightarrow \\ \operatorname{rank}{\{\nabla_{z}g_i(u,z),\nabla_{z}g_j(u,z)\}}=2
        \end{gathered}
    \end{equation}
    holds at $u$, then a sample pair $\{\boldsymbol{v}_i\}_{i=1}^N$, $\{z_i\}_{i=1}^N$ satisfies the assumptions in Proposition~\ref{Prop:differentiability} with probability one.
\end{proposition}

\begin{proof}
    We only need to show that a single sample pair $(\boldsymbol{v},z)$ satisfies the assumptions with probability 1. We first show that assumption~\ref{Ass:index} is satisfied, i.e., $g_j(u,\tilde{\xi})\neq 0$ for all $j \in \{1,\dots,M\}$ with probability 1. Since $z$ is drawn from the remainder $\xi_R$, and $g_j(u,\cdot)$ is an affine functional, the evaluation $g_j(u,\tilde{\xi})=g_j(u,\bar{\xi}+z)$ is a one-dimensional random variable. If the variance of $g_j(u,\tilde{\xi})$ is zero, then $g_j(u,\tilde{\xi})$ is deterministic and almost surely equals $g_j(u,\bar{\xi})$. By~\eqref{Eq:neg_gj}, $g_j(u,\bar{\xi})<0$, so it cannot equal $0$. Alternatively, if the variance of $g_j(u,\tilde{\xi})$ is positive, then $g_j(u,\tilde{\xi})$ is a non-degenerate one-dimensional Gaussian random variable, which implies that $\bbP(g_j(u,\tilde{\xi})=0)=0$. Taking the union over all $M$ constraints, we conclude that $g_j(u,\tilde{\xi}) \neq 0$ for all $j\in\{1,\dots,M\}$ with probability 1.

    Next, we show that for all $z$ satisfying assumption~\ref{Ass:index}, assumptions~\ref{Ass:inout} and \ref{Ass:unique} are satisfied for $\boldsymbol{v}$ almost surely. We only need to show that for any $i,j\in\{1,\dots,M\}$ with $i\neq j$, the set of rays intersecting both constraints at the same finite distance,
    \begin{equation*}
        M_{i,j}:=\{\boldsymbol{v}\in\bbS^{K-1}:g_i(u,\tilde{\xi}+rL_K\boldsymbol{v})=g_j(u,\tilde{\xi}+rL_K\boldsymbol{v})=0\text{ for some }r>0\},
    \end{equation*}
    satisfies $\mu_{\cU}(M_{i,j})=0$. Define
    \begin{equation*}
        \tilde{g}_{i,j}(u,\tilde{\xi}+rL_K\boldsymbol{v}):=\max_{k\in\{i,j\}}{\{-\sgn{(g_k(u,\tilde{\xi}))}g_k(u,\tilde{\xi}+rL_K\boldsymbol{v})\}},
    \end{equation*}
    then $\tilde{g}_{i,j}(u,\tilde{\xi})<0$, and $\tilde{g}_{i,j}$ satisfies the R2CQ property~\eqref{Eq:R2CQ}. Therefore, by~\cite[Lem.~4.3]{AcHe-SIAM17}, $\mu_{\cU}(M_{i,j})=0$. This ensures that rays almost never intersect two constraint boundaries at the exact same finite distance, which guaranties that almost surely, finite values of $\tilde{\rho}_{\mathrm{in}}$ and $\tilde{\rho}_{\mathrm{out}}$ are achieved at unique constraint indices (satisfying assumption~\ref{Ass:unique}) and $\tilde{\rho}_{\mathrm{in}}$ cannot equal $\tilde{\rho}_{\mathrm{out}}$ unless both are infinite (satisfying assumption~\ref{Ass:inout}). Thus, the assumptions in Proposition~\ref{Prop:differentiability} are satisfied by a single sample pair $(\boldsymbol{v},z)$ with probability 1, and the proof is complete.
\end{proof}

\subsection{Computational complexity of {{\footnotesize HI}SRD}}
The evaluation of $\tilde\varphi_N$ using \hiSRD\ entails: (1) drawing independent samples from the $K$-dimensional SRD subspace $\mathscr H_K$ along with samples $z$ from the remainder space $\mathscr H_R$; (2) computing the entry and exit distances $\tilde\rho_\text{in}$ and $\tilde\rho_\text{out}$ as defined in \eqref{Eq:rhoModified}; and (3) evaluating \eqref{Eq:phi_u_dis2}, which requires computing $F_{\chi_K}$, the cumulative distribution function of the Chi distribution. Finally, (4), the gradient computation, as specified in \eqref{Eq:GradDis}, reuses $\tilde\rho_\text{in}$ and $\tilde\rho_\text{out}$, involves evaluation of the $\chi$-probability density function $f_{\chi_K}$, and consists primarily of negligible additional operations.

For the affine constraints considered in this work, the entry and exit distances $\tilde\rho_\text{in}$ and $\tilde\rho_\text{out}$ in step (2) can be efficiently computed by maximization and minimization operations on suitable index sets. The SRD (and \hiSRD) extends to non-affine constraint functions $g_i(u,\cdot)$, for example, convex ones. In such cases, computing these distances may require iterative algorithms, which can substantially increase the computational cost. Steps (3) and (4) usually do not dominate the computation, since evaluating special functions like those in the $\chi$-distribution is computationally inexpensive. Consequently, the dominant computational cost is typically associated with step (1), i.e., generating independent samples from $\xi$. This cost depends heavily on the specific setting of the problem.

For the elliptic PDE control problem with a constant PDE operator $A$ as in \eqref{eq:redstate}, one natural approach is to use a KL expansion for the random function $\xi(\omega)$ and then solve the linear PDE for each SRD sample (which can be done efficiently, for example, by reusing a decomposition of $A$). As an alternative, one may employ a KL expansion in the state variable $y$, which already incorporates the PDE solution. Such a KL expansion may exhibit more rapidly decaying coefficients and enables direct sampling from the distribution of the states, which are subject to the probabilistic constraint (see also \cite[Fig.~1]{HeStWe-SIAM25}). This KL expansion of the state variable can be precomputed once and reused throughout the optimization.

In contrast, for Gaussian process regressions, the covariance matrix and the corresponding transform operator $A(u)$ in \eqref{Eq:GPtransform} depend on the kernel hyperparameters $u$. Since $u$ changes in every optimization iteration, one cannot rely on a precomputed expansion or a fixed matrix factorization. Instead, we recompute the covariance matrix and its factorization in every iteration, and apply $A(u)$ to the matrix consisting of reference samples $\xi_{\mathrm{ref}}$ through a matrix-matrix multiplication, which amounts to the main computational cost.

Overall, all computations needed to approximate the probability function and its gradient scale linearly with $N$, the number of Monte Carlo samples, and are largely independent of $K$, the dimension of the subspace $\mathscr H_K$ on which the SRD is implemented.

\section{Numerical results for the control problem}\label{SEC:NumCtrl}
We now present a numerical example for a PDE control problem of the form \cref{ex:optcont}. We consider the physical domain $\D=(0,1)^2\subset \mathbb R^2$ and divide the boundary into $\partial \D_2=\{0\}\times [0,1]$ and $\partial \D_1=\partial \D\setminus \partial \D_2$.
We consider the risk-neutral, tracking-type optimal control problem
\begin{equation}\label{eq:ex1-J}
        \begin{aligned}
        & \underset{{u\in L^2(\D), \,y\in X}}{\text{minimize}}\:
        & & 
        \frac 12 \int_\Omega
  \int_\D (y(\omega) - y_d)^2 \,d x\,d\mu
  + \frac\alpha 2 \int_\D u^2\,d x 
        \end{aligned}
\end{equation}
subject to the governing equations with uncertain Neumann data $\partial\D_2$:
\begin{alignat}{2}
  -\Delta y(\omega) &=
  f + u \qquad &&\text{in } \D, \label{eq:state_linear1}\\
  y(\omega)  &= 0  &&\text{in }\partial\D_1, \label{eq:state_linear2}\\
  \nabla y(\omega)\cdot{n} &= \xi(\omega) &&\text{in
  }\partial\D_2, \label{eq:state_linear3}
  \end{alignat}
and the joint state chance constraints \eqref{eq:chance}. 
Here, $y_d = \frac{1}{10} \cos(2\pi
x_1)\sin(2\pi x_2)$, $f\equiv 0$, $\alpha=10^{-5}$. 
The uncertain parameter field enters as Neumann data on the one-dimensional domain
$\partial \D_2$. This data follows an infinite-dimensional Gaussian
distribution with mean $\xi_0 \equiv 0$, and a covariance operator 
given by the inverse elliptic PDE operator $\mathcal C_0 =
\gamma(-\partial_{{x_2 x_2}})^{-1}$, with homogeneous Dirichlet conditions at the boundary of $\partial \D_2$, i.e., at the
two points $(0,0)$ and $(0,1)$, and with $\gamma=4$. This covariance
operator has the eigenfunctions $\sin(k x_2)$, $k=1,2,\ldots$ with
corresponding eigenvalues $4(k\pi)^{-2}$. 

This problem has also been considered in \cite{HeStWe-SIAM25}, where it is proven that the probability function is well-defined if one understands the chance constraint \eqref{eq:chance} in an almost everywhere sense, and that the PDE optimal control problem has a unique solution. Moreover, it is shown that the integration over $\Omega$ in \eqref{eq:ex1-J} can be performed analytically and that the state $y$ is an affine function of $u$ and $\xi$. Thus, this problem can be brought into the form \eqref{eq:Jopt}, \eqref{Eq:phi_u}.

\subsection{Numerical setup}
The governing equations \eqref{eq:state_linear1}, \eqref{eq:state_linear2} and \eqref{eq:state_linear3} are discretized using finite differences on a regular grid with $64\times 64$ points. The bound constraints are enforced at the same points and thus $M=4096$. Linear systems are solved with a direct sparse solver, and a standard adjoint method is used to compute the gradient with respect to the distributed control $u\in L^2(\D)$.

\subsection{Estimation of probability function}
Here, we only focus on the probability estimation for the nominal control $u=\frac 15 \sin(2\pi x_1)\cos(\pi x_2)$ with state bounds $\underline y\equiv -0.3$, $\bar y\equiv 0.3$ in \eqref{eq:chance}.
In \cite{HeStWe-SIAM25}, the bias resulting from truncation of the KL expansion of the Gaussian random field were illustrated. Here, in \cref{fig:control}, we show the superior behavior of the proposed method compared to KL expansion truncation combined with finite-dimensional SRD. As can be seen in the figure, the proposed method is an unbiased estimator. The bias of the finite-dimensional SRD estimators based on truncation of the KL expansion starts to dominate the variance of the estimator when $N$ is sufficiently large. In \cref{fig:control}, we choose $K=10$ for the hybrid estimator. The results for $K=5,15,20$ in the hybrid method are not shown but would be very similar. All SRD samplers used a quasi-Monte Carlo sampling of the uniform distribution on the sphere, which has been observed to reduce the variance in SRD estimation. Choosing a standard Monte Carlo method to sample the unit sphere results in qualitatively similar results but generally larger variance--see for instance the comparison in \cite{HeStWe-SIAM25}.

\subsection{Optimal control under joint chance state constraints}
Estimating $\tilde\varphi_N(\cdot)$ and its gradient $\nabla\tilde\varphi_N(\cdot)$ enables us to compute derivatives of the constraint with respect to the control $u$.
The corresponding gradient field for the reference parameters is shown on the right in \cref{fig:control}. To solve the chance-constrained control problem, one may employ a sequential quadratic programming (SQP) algorithm. For a detailed discussion on the optimal control problem, we refer to \cite{HeStWe-SIAM25}, where $\tilde\varphi_N(\cdot)$ is approximated by using a truncated KL expansion combined with a finite-dimensional SRD. In the next section, we proceed with an extensive numerical study of the new method applied to constrained Gaussian process regression.

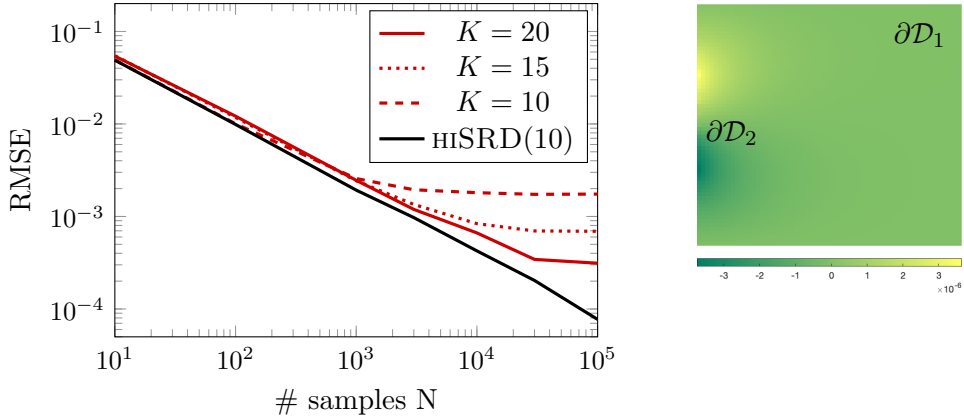
\begin{figure}[tbp]\centering
  \begin{tikzpicture}\centering
    \begin{loglogaxis}[
        xlabel=\# samples N,
        ylabel={RMSE},
        xmin=10, xmax=1e5, ymax=0.2, ymin=5e-5,
        height=6cm, width=8cm,
        ticklabel style = {font=\small}]
      \addplot[red!80!black, very thick] table [x=N,y=dim20] {data/biyupylow03kXX_rmse.txt};
      \addlegendentry{$K=20$}
      \addplot[red!80!black, very thick, dotted] table [x=N,y=dim15] {data/biyupylow03kXX_rmse.txt};
      \addlegendentry{$K=15$}
      \addplot[red!80!black, very thick, dashed] table [x=N,y=dim10] {data/biyupylow03kXX_rmse.txt};
      \addlegendentry{$K=10$}
      \addplot[black, very thick] table [x=N,y=hiSRD10] {data/biyupylow03kXX_rmse.txt};
      \addlegendentry{\hiSRD(10)}
    \end{loglogaxis}
  \node at (9.5,2.5) {\includegraphics[width=3.5cm]{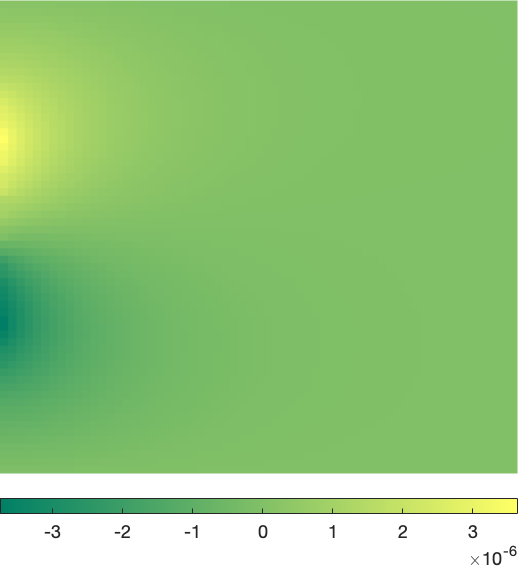}};
\node at (8.2,2.7) {$\partial\D_2$};
\node at (10.7,4) {$\partial\D_1$};
  \end{tikzpicture}

  \caption{Left: Root mean squared error (RMSE, $y$-axis) of probability estimation \eqref{eq:chance} for different number $N$ of MC samples ($x$-axis). Compared are the finite-dimensional SRD for various KL-mode truncations $K$, and the proposed hybrid method. The ``exact'' probability, computed with $10^8$ standard Monte Carlo samples, is $p\approx 0.648926$. Right: Gradient at the nominal control, which is largest close to $\partial \D_2$, where the stochastic Neumann boundary force is applied.}\label{fig:control}
\end{figure}
\section{Numerical results for Gaussian process regression}\label{SEC:NumGPR}
We now apply \hiSRD\ to kernel parameter optimization in Gaussian process regression subject to joint chance constraints, as outlined in~\cref{ex:GPR}. We assume that the covariance function of the prior process is 
\begin{equation}\label{Eq:GPkernel}
   \mathcal K_u(x,x')=\sigma^2\exp\left(\frac{-\|{x-x'}\|_2^2}{2l^2}\right)+\sigma_n^2\delta_{x,x'},
\end{equation}
where $u:=(\log{l},\log{\sigma},\log{\sigma_n})$ are the kernel parameters to optimize, and the prior mean is $\xi_0\equiv 0$. Choosing $\cD=[0,1]$ we consider noiseless observations of the following function, also used in~\cite{PeYaZh-TAML20}:
\begin{equation*}
    f(x):=\frac{1}{1+(10x)^4}+\frac{1}{2}\exp{\Big[ -100 \Big( x-\frac{1}{2} \Big)^2 \Big]},
\end{equation*}
at positions
\begin{equation*}
    \big\{ x^{(i)} \big\}_{i=1}^7=\big\{ {(j-1)/}{5}+\epsilon_j \big\}_{j=1}^6\cup \big\{ {1}/{2} \big\},
\end{equation*}
where $\epsilon_j\sim\cN(0,0.03^2)$ for $j=2,3,4,5$, and $\epsilon_1=\epsilon_6=0$. That means, we set the observation as $\boldsymbol{y}=f(\boldsymbol{X})$, where $\boldsymbol{X}=(x^{(1)},\dots,x^{(7)})$.

Recall that $K_u=\cK_u(\boldsymbol{X},\boldsymbol{X})$ denotes the covariance matrix for the points in $\boldsymbol X$. We also use $k_u(x):=\cK_u(\boldsymbol{X},x)-\sigma_n^2\delta_{\boldsymbol{X},x}\in\bbR^N$ to denote the noise-free cross-covariance vector between $\boldsymbol{X}$ and other points $x\in\cD$. Conditioned on the observations, the posterior Gaussian process is $\xi_{\mathrm{post}}(u)=\cN(\xi^\star(u),\cK^\star_u)$, where
\begin{gather*}
    \xi^\star(x;u)=\xi_0(x)+k_u(x)^T K_u^{-1}\boldsymbol{y}, \quad
    \cK^\star_u(x,x')=\sigma^2\exp\Big(\frac{-\norm{x-x'}_2^2}{2l^2}\Big)-k_u(x)^T K_u^{-1}k_u(x').
\end{gather*}
We use $\xi_{\mathrm{ref}}=\xi_{\mathrm{post}}(u_0)-\xi^\star(u_0)$ as the reference Gaussian process with fixed kernel parameters $u_0=(-3,-3,-10)$. We take $A(u) = S_u S_{u_0}^{-1}$ as the transform in~\eqref{Eq:GPtransform}, where $S_u$ is an operator that satisfies $S_u S_u^* = \cK^\star_u$. 
In numerical implementation, we compute $S_u$ using the Cholesky decomposition of the discretized covariance matrix.

\subsection{Numerical setup}
We use $n$ uniformly distributed points to discretize $\cD$. Unless otherwise specified, we set $n=1024$. We evaluate the state constraint on the same grid and use a one-sided bound, yielding $M=n$. Unless otherwise specified, we use $K=10$ for the low-rank truncation dimension. After discretizing the Gaussian process, we use the Cholesky decomposition of the covariance matrix to generate samples of $\xi_{\mathrm{ref}}$ and compute $A(u)$. Both MC and QMC (quasi-Monte Carlo) are used to generate samples for the uniform distribution on the sphere and for the remainder. To obtain samples of the uniform distribution on the sphere, we first generate MC or QMC samples from the standard Gaussian distribution in $\bbR^K$, which we then project onto the unit sphere. For QMC, we use the Halton sequence.

In addition to the chance constraint~\eqref{Eq:ChanceConstrGPR}, we employ the following constraint (also used in~\cite{PeYaZh-TAML20})  on the distance of the posterior mean from the observations, ensuring that the posterior Gaussian process fits the data sufficiently well:
\begin{equation}\label{Eq:HeurConstr}
    \norm{\mathbb{E}[\xi_\text{post}(\boldsymbol{X})]-\boldsymbol{y}}_\infty \le \epsilon,
\end{equation}
where $\epsilon>0$ is a fixed small constant.  This prevents the noise variance $\sigma_n$ from diverging during optimization. Large $\sigma_n$ are not relevant likelihood minimizers because the corresponding posterior process does not adequately fit the data $\boldsymbol{y}$. In our experiment, we set $\epsilon=0.01$.

We use {\tt tinygp}~\cite{tinygp}, a lightweight library built on top of {\tt JAX}, to perform operations on Gaussian processes such as conditioning. We also use {\tt JAX} to automatically compute the gradients of the probability estimator $\tilde\varphi_N(u)$ with respect to the kernel parameters for optimization. While in this work we chose the quadratic exponential kernel \eqref{Eq:GPkernel}, the use of {\tt tinygp} makes it easy to adopt more complex kernels, or mixtures of kernels with an arbitrary number of parameters, without increasing the computational cost of calculating gradients.

\subsection{Estimation of probability function and its gradient}\label{subsec:numerical-variance}
Next, we examine how accurately the various methods approximate the probability function $\varphi(u)$ for the lower bound $\underline{\xi}=-0.12$ and the reference kernel parameters $u=(-3,-3,-10)$. For this choice of parameters, we obtain $\varphi(u)\approx 0.977$, i.e., only about 2\% of the Gaussian process samples violate the joint chance constraints.

\begin{figure}[tbp]
    \centering
        \begin{tikzpicture}
            \begin{loglogaxis}[
                xlabel={truncation dimension},
                ylabel={RMSE},
                xmin=1, xmax=1000,
                ymin=1e-4, ymax=3e-2,
                width=8cm, height=6.5cm,
                legend pos=north east,
                legend cell align={left},
                legend style={font=\small},
                grid style=dotted,
            ]
            \addplot[name path=MC, color=black!50!white, mark=none, thick]
                table[x=k, y=MC] {data/RMSE_prob_GPR.txt};
                \addlegendentry{MC}
            \path[name path=axis] (axis cs:1,1e-4) -- (axis cs:1000,1e-4);
            \addplot[black, opacity=0.04, on layer=axis background, forget plot] fill between[of=MC and axis];
            \addplot[color=black!30!blue, mark=none, very thick]
                table[x=k, y=hiSRD_MC] {data/RMSE_prob_GPR.txt};
                \addlegendentry{\hiSRD-MC}
            \addplot[color=black!30!green, mark=none, very thick]
            table[x=k, y=hiSRD_QMC] {data/RMSE_prob_GPR.txt};
            \addlegendentry{\hiSRD-QMC}
            \addplot[color=blue!40, mark=none, thick, dashed]
                table[x=k, y=SRD_MC] {data/RMSE_prob_GPR.txt};
                \addlegendentry{SRD-MC}
            \addplot[color=green!60, mark=none, thick, dash dot]
            table[x=k, y=SRD_QMC] {data/RMSE_prob_GPR.txt};
            \addlegendentry{SRD-QMC}
            \end{loglogaxis}
        \end{tikzpicture}
    \hspace{.04\textwidth}
        \centering
        \begin{tikzpicture}
            \begin{loglogaxis}[
                xlabel={truncation dimension},
                xmin=1, xmax=20,
                ymin=1e-4, ymax=3e-2,
                width=4.5cm, height=6.5cm,
                xtick={1, 5, 10, 20},
                xticklabels={1, 5, 10, 20},
                minor xtick={2, 3, 4, 6, 7, 8, 9},
                legend pos=north east,
                legend cell align={left},
                legend style={font=\small},
                grid style=dotted,
            ]
            \path[name path=axis] (axis cs:1,1e-4) -- (axis cs:1000,1e-4);
            \addplot[color=black!30!blue, mark=none, very thick]
                table[x=k, y expr=sqrt(\thisrow{V_total})] {data/RMSE_split.txt};
                \addlegendentry{$V_{\text{total}}$}
            \addplot[color=red!60!blue, mark=none, thick, dash dot dot]
            table[x=k, y expr=sqrt(\thisrow{V_SRD})] {data/RMSE_split.txt};
            \addlegendentry{$V_{\text{SRD}}$}
            \addplot[color=orange!70!blue, mark=none, thick, densely dashed]
            table[
                x=k, 
                y expr=sqrt(\thisrow{V_rem}), 
            ] {data/RMSE_split.txt};
            \addlegendentry{$V_{\text{rem}}$}
            \end{loglogaxis}
        \end{tikzpicture}
    \caption{Gaussian process regression: Shown on the left is a comparison of the root mean squared error (RMSE; $y$-axis) of probability estimation with $N=10^4$ samples versus KL dimension $K$ ($x$-axis) for finite-dimensional SRD (dashed) and for the proposed method (solid), with both MC and QMC sampling. Shown for comparison is the result with standard Monte Carlo (gray horizontal line). Shown on the right are, for $N=500$ samples, the individual components discussed in \cref{subsec:variance-analysis} contributing to the variance of \hiSRD-MC sampling. The larger variance compared to the left figure is due to the smaller $N$.}
    \label{Fig:RMSE_prob_GPR}
\end{figure}
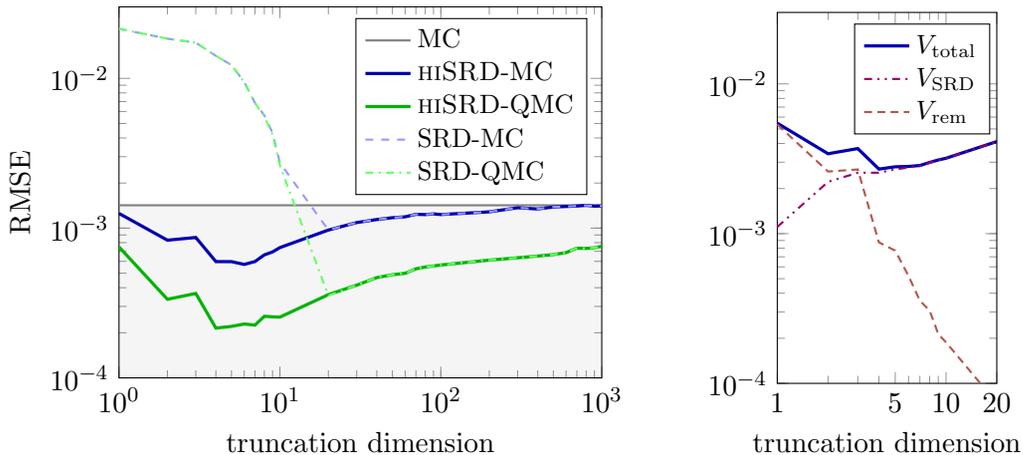

In \cref{Fig:RMSE_prob_GPR}, which is an extension of \cref{fig:intro} shown in the introduction, we compare the root mean square error obtained from 100 repeats with $N=10^4$ samples each for standard MC sampling, finite-dimensional SRD and \hiSRD. We use $10^7$ samples from the standard MC method to compute a highly accurate reference probability. As can be seen in the left subfigure, the proposed \hiSRD\ estimator is unbiased and consistently achieves lower RMSE, while the finite-dimensional SRD estimator has a large bias for small $K$. This variance reduction is more significant when combined with QMC.

On the right in \cref{Fig:RMSE_prob_GPR}, we show the individual contributions to the variance of the \hiSRD\ estimator, as discussed in~\cref{subsec:variance-analysis}. To separate the expected SRD variance $V_{\text{SRD}}$ from the remainder variance $V_{\text{rem}}$, we draw $N_{\text{rem}}=500$ samples from the remainder and, for each, draw $N_{\text{SRD}}=100$ independent spherical samples and estimate the RMSE from 100 repeats. Let $V_{\text{total}}'$ denote the empirical total variance of this nested estimator.
By computing the sample variance of the individual spherical evaluations, which estimates the theoretical $V_{\text{SRD}}$ defined in~\eqref{eq:VSRD}, we isolate the empirical remainder variance via $V_{\text{rem}} = V_{\text{total}}' - \frac{1}{N_{\text{SRD}}}V_{\text{SRD}}$. The total variance for the standard setting $N_{\text{SRD}}=1$ is then recovered by $V_{\text{total}} = V_{\text{rem}} + V_{\text{SRD}}$. The figure shows the respective contributions of the spherical and remainder variances to the total variance for different truncations $K$. As can be seen, the remainder variance $V_{\text{rem}}$ decays quickly as $K$ increases, and $V_{\text{SRD}}$ dominates the variance for moderate $K$.

We also perform a similar variance comparison to \cref{Fig:RMSE_prob_GPR} for the gradient estimation with respect to the kernel parameters. These results are shown in \cref{Fig:RMSE_grad_GPR}, where the plotted RMSE represents the expected $\ell_2$-norm of the difference between the estimated gradient and a highly accurate reference gradient obtained by using $10^7$ samples from \hiSRD. Note that the standard MC estimator is not differentiable, so this method is excluded from this comparison. Similar to the estimation $\tilde \varphi_N(u)$, \hiSRD\ provides a highly accurate gradient estimation, with QMC sequences further reducing the variance.

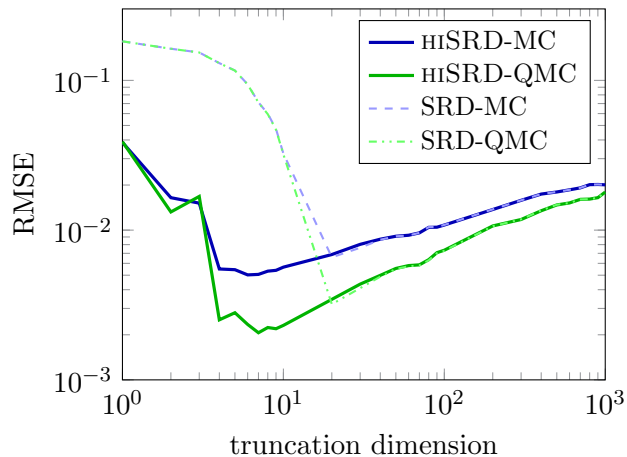
\begin{figure}[tbp]
    \centering
    \begin{tikzpicture}
        \begin{loglogaxis}[
            xlabel={truncation dimension},
            ylabel={RMSE},
            xmin=1, xmax=1000,
            ymin=1e-3, ymax=3e-1,
            width=8cm, height=6.5cm,
            legend pos=north east,
            legend cell align={left},
            legend style={font=\small},
            grid style=dotted,
        ]
        \addplot[color=black!30!blue, mark=none, very thick]
            table[x=k, y=hiSRD_MC] {data/RMSE_grad_GPR.txt};
            \addlegendentry{\hiSRD-MC}
        \addplot[color=black!30!green, mark=none, very thick]
           table[x=k, y=hiSRD_QMC] {data/RMSE_grad_GPR.txt};
           \addlegendentry{\hiSRD-QMC}
        \addplot[color=blue!40, mark=none, thick, dashed]
            table[x=k, y=SRD_MC] {data/RMSE_grad_GPR.txt};
            \addlegendentry{SRD-MC}
        \addplot[color=green!60, mark=none, thick, dash dot dot]
           table[x=k, y=SRD_QMC] {data/RMSE_grad_GPR.txt};
           \addlegendentry{SRD-QMC} 
        \end{loglogaxis}
    \end{tikzpicture}
    \caption{Gaussian process regression: Companion plot to \cref{Fig:RMSE_prob_GPR}, but for estimation of the gradient. Probability estimation via standard MC sampling is non-differentiable and therefore not included, since its RMSE is unbounded.
     }
    \label{Fig:RMSE_grad_GPR}
\end{figure}

\subsection{Optimal kernel parameters under chance constraints}
To solve the constrained optimization problem, we use a built-in solver in {\tt SciPy} implementing the trust-region interior point method \cite{ByHrNo-SIAM99}. We approximate the Hessian of the chance constraint using symmetric-rank-1 updates.

We optimize the GP kernel parameters using a one-sided lower bound $\underline{\xi}$ and set $p = 0.95$. In~\cref{Fig:KernelOpt}, we present the posterior Gaussian process with the resulting optimal kernel parameters, subject to the chance constraint for two different lower bounds $\underline{\xi}$, which we choose as constants. Because $p$ is close to 1, the bound constraint is satisfied by most posterior samples. It can be observed that, in order to keep the joint chance constraint satisfied under a stricter lower bound, the optimization algorithm reduces the variance of the posterior process. This effect is also visible in \cref{tab:GP}, where we report the optimized parameter values $u$ for several (constant) lower bounds. As the lower bound is increased, not only does the variance $\sigma$ decrease, but the length-scale parameter $l$ also becomes smaller. The optimization 
chooses rather small values of $\sigma_n$, indicating that the posterior process fits the data well. These small values for the optimized $\sigma_n$ may be a consequence of imposing \eqref{Eq:HeurConstr}.

\begin{table}[htbp]
    \centering
    \caption{Optimal kernel parameters and objective function \eqref{Eq:ObjFunc} values for different lower bounds $\underline \xi$. Results use \hiSRD\ with $N=10^4$ samples. $\underline \xi=-\infty$ corresponds to the unconstrained Gaussian process regression.} \label{tab:GP}
    \begin{tabular}{c|cccc}
        $\underline\xi$ &$l$ &$\sigma$ &$\sigma_n$ & $\mathcal J(u)$ 
        \\[0.2em]
        \hline 
        $-\infty$ & $\num{1.16e-01}$ & $\num{4.53e-01}$ & $\num{5.55e-05}$ & $\num{3.19e+00}$\\[0.2em]
        $-0.2$ & $\num{8.54e-02}$ & $\num{1.18e-01}$ & $\num{5.57e-05}$ & $\num{3.60e+01}$\\[0.2em]
        $-0.05$ & $\num{6.94e-02}$ & $\num{3.07e-02}$ & $\num{5.71e-05}$ & $\num{6.49e+02}$\\[0.2em]
        $-0.01$ & $\num{6.71e-02}$ & $\num{7.10e-03}$ & $\num{1.23e-05}$ & $\num{1.24e+04}$\\[0.2em]
    \end{tabular}
    \label{Table:kernel parameter}
\end{table}

\begin{figure}[tbp]
    \centering
    \includegraphics[width=0.95\linewidth]{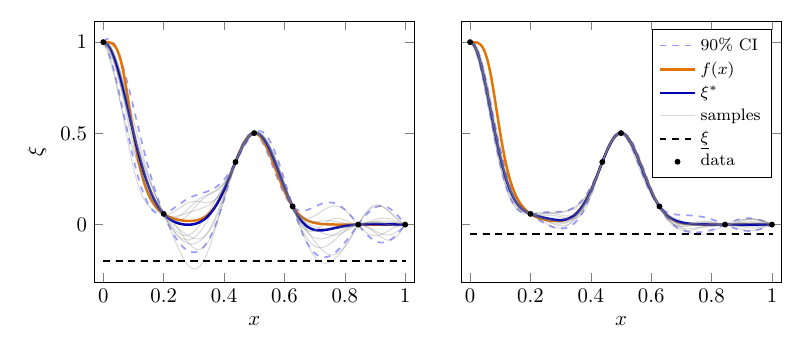}
    \caption{Posterior Gaussian processes with optimal kernel parameters, with different lower bounds $\underline{\xi}$ for the chance constraint. Each figure shows the posterior mean (blue), the 90\% credible interval of the Gaussian process (dashed), and 10 samples from the Gaussian process (gray). Left: $\underline{\xi}\equiv-0.2$. Right: $\underline{\xi}\equiv-0.05$.}
    \label{Fig:KernelOpt}
\end{figure}

\section{Conclusions}
Approximating probabilistic functions over high- or infinite-dimensional distributions can be very challenging. For finite-dimensional elliptical distributions, the spherical-radial decomposition (SRD) has been shown to be highly effective for the differentiable approximation of probabilistic functions that occur in risk-averse stochastic optimization. However, because the standard finite-dimensional SRD relies on unit spheres in finite-dimensional spaces, it was previously unclear whether this approach can be extended to infinite-dimensional settings. This paper demonstrates that such a generalization is indeed possible. Rather than working with sequence spaces, as one might first expect, our proposed method combines SRD on suitable subspaces with standard Monte Carlo sampling on the subspace complement. We presented and analyzed this generalized framework, demonstrating that it preserves many of the advantageous characteristics of finite-dimensional SRD. In numerical case studies involving an optimal control problem governed by a PDE and a constrained Gaussian process regression, we illustrated the method's accuracy and discussed its computational complexity for approximating joint chance constraints. 
Future work could concentrate on applying the chance-constrained Gaussian process regression approach to problems in higher spatial dimensions. Another promising direction would be to generalize the method to nonlinear mappings (for example, general convex maps) from the uncertain parameters to the state variable.

\section*{Acknowledgments} The authors thank Timo Schorlepp for his helpful comments on an earlier version of this manuscript.

\bibliographystyle{siamplain}
\bibliography{refs-HiSRD}

\appendix
\section{Proof of \texorpdfstring{\cref{thm:estimation-degeneration}}{}}\label{app:1}
For simplicity, we use $x\lesssim y$ to denote that $x\leq Cy$ for a uniform constant $C>0$ independent of $K$ and $c$. By \cref{LEMMA:mean_K} and \eqref{Eq:Wk} we have
\begin{equation}\label{Eq:W_diff}
    \begin{aligned}
        W_{\mathrm{MC}}(c)-W_K(c) &= \bbE_{\tau_{r,K}}\Big[ F_{\chi_K}\Big( \frac{1}{\sqrt{\tau_{r,K}}}c \Big)\Big( 1-F_{\chi_K}\Big( \frac{1}{\sqrt{\tau_{r,K}}}c \Big) \Big) \Big] \\
        &= \bbE_{\tau_{r,K}}\Big[ F_{\chi_K^2}\Big( \frac{1}{\tau_{r,K}}c^2 \Big)\Big( 1-F_{\chi_K^2}\Big( \frac{1}{\tau_{r,K}}c^2 \Big) \Big) \Big],
    \end{aligned}
\end{equation}
where $\chi_K^2$ is the $K$-dimensional $\chi$-squared distribution. Since the square of a standard normal distribution has mean 1 and variance 2, and all moments of $\chi_K$ are finite, by the Berry-Esseen theorem~\cite[Sec.~3.4.4]{Durrett-book19}, we obtain
\begin{equation*}
    \bigg| F_{\chi_K^2}\left( \frac{1}{\tau_{r,K}}c^2 \right) - \Phi \bigg( \frac{\frac{c^2}{\tau_{r,K}}-K}{\sqrt{2K}} \bigg) \bigg| \lesssim \frac{1}{\sqrt{K}},
\end{equation*}
where $\Phi$ is the CDF of the standard normal distribution. Therefore, by~\eqref{Eq:W_diff},
\begin{equation}\label{Eq:W_diff2}
    \begin{aligned}
        W_{\mathrm{MC}}(c)-W_K(c) &= \bbE_{\tau_{r,K}}\bigg[ \Phi \bigg( \frac{\frac{c^2}{\tau_{r,K}}-K}{\sqrt{2K}} \bigg)\bigg( 1-\Phi \Big( \frac{\frac{c^2}{\tau_{r,K}}-K}{\sqrt{2K}} \Big) \Big) \bigg] + \cO \Big( K^{-\frac{1}{2}} \Big)\\
        &=: \bbE_{\tau_{r,K}}\left[ \Phi \left( s(\tau_{r,K}) \right)\left( 1-\Phi \left( s(\tau_{r,K}) \right) \right) \right] + \cO \Big( K^{-\frac{1}{2}} \Big),
    \end{aligned}
\end{equation}
where
\begin{equation}
    s(t):=\frac{{c^2}/{t}-K}{\sqrt{2K}}=\frac{c^2}{\sqrt{2K}}\frac{1}{t}-\sqrt{{K}/{2}}.
\end{equation}
We then rewrite the first term in~\eqref{Eq:W_diff2} in integral form and split it into two parts:
\begin{equation*}
    \begin{aligned}
        \bbE_{\tau_{r,K}}\left[ \Phi \left( s(\tau_{r,K}) \right)\left( 1-\Phi \left( s(\tau_{r,K}) \right) \right) \right] &= \int_0^1 \Phi \left( s(t) \right)\left( 1-\Phi \left( s(t) \right) \right) f_{\tau_{r,K}}(t)dt \\
        &= \int_{|s(t)|\leq \sqrt{\log{K}}} \Phi \left( s(t) \right)\left( 1-\Phi \left( s(t) \right) \right) f_{\tau_{r,K}}(t)dt \\
        &\!\!\!\! + \int_{|s(t)|\geq \sqrt{\log{K}}} \Phi \left( s(t) \right)\left( 1-\Phi \left( s(t) \right) \right) f_{\tau_{r,K}}(t)dt
        =: I_1+I_2.
    \end{aligned}
\end{equation*}
Here, $f_{\tau_{r,K}}$ is the PDF of $\mathrm{Beta} \left( \frac{r}{2},\frac{K-r}{2} \right)$,
\begin{equation*}
    f_{\tau_{r,K}}(x)={B\Big( \frac{r}{2},\frac{K-r}{2} \Big)}^{-1}x^{\frac{r}{2}-1}(1-x)^{\frac{K-r}{2}-1},\quad 0\leq x\leq 1.
\end{equation*}
Using the following inequality of $\Phi(x)$
\begin{equation}\label{Eq:asy_Phi}
    \min{\{\Phi(x), 1-\Phi(x)\}}\leq {1}/{(\sqrt{2\pi}|x|)}e^{-\frac{x^2}{2}},
\end{equation}
which holds for all $x\neq 0$, we can bound $I_2$ by
\begin{equation}\label{Eq:bd_I2}
    I_2\lesssim {1}/{\sqrt{K}}.
\end{equation}

To bound $I_1$, we first use a change of variables to obtain:
\begin{equation}
    I_1=\int_{|s|\leq \sqrt{\log{K}}}\Phi \left( s \right)\left( 1-\Phi \left( s \right) \right) f_{\tau_{r,K}}(t(s))\cdot \Big( \frac{c^2}{\sqrt{2K}}\big( s+\sqrt{{K}/{2}} \big)^{-2} \Big) ds.
\end{equation}
Then we derive an inequality for $f_{\tau_{r,K}}$ to bound this intergal. Using Stirling's formula for Gamma functions, we have
\begin{equation*}
    {B\Big( \frac{r}{2},\frac{K-r}{2} \Big)}^{-1}=\frac{\Gamma(\frac{K}{2})}{\Gamma(\frac{r}{2})\Gamma(\frac{K-r}{2})}\lesssim K^{\frac{r}{2}}.
\end{equation*}
Since $|s|\leq \sqrt{\log{K}}$ implies $t(s)\sim {c^2}/{K}$, we have 
\begin{equation*}
    f_{\tau_{r,K}}(t(s))\lesssim K^{\frac{r}{2}} \left( \frac{c^2}{K} \right)^{\frac{r}{2}-1} \left( 1-\frac{c^2}{K} \right)^{\frac{K-r}{2}-1} \lesssim K c^{r-2}e^{-\frac{c^2}{2}}.
\end{equation*}
Since $c^re^{-\frac{c^2}{2}}$ is uniformly bounded for all $c>0$, and $|s|\leq \sqrt{\log{K}}$, it holds that
\begin{equation*}
    f_{\tau_{r,K}}(t(s)) \left( \frac{c^2}{\sqrt{2K}}\left( s+\sqrt{{K}/{2}} \right)^{-2} \right)\lesssim K c^{r}e^{-\frac{c^2}{2}}K^{-\frac{3}{2}}\lesssim\frac{1}{\sqrt{K}}.
\end{equation*}
By~\eqref{Eq:asy_Phi}, we have $\int_\bbR \Phi \left( s \right)\left( 1-\Phi \left( s \right) \right)ds <\infty$, and thus $I_1$ can be bounded by
\begin{equation}\label{Eq:bd_I1}
    I_1\leq \int_\bbR \Phi \left( s \right)\left( 1-\Phi \left( s \right) \right)ds\cdot\sup_{|s|\leq \sqrt{\log{K}}}\left| f_{\tau_{r,K}}(t(s))\cdot \left( \frac{c^2}{\sqrt{2K}}\left( s+\sqrt{{K}/{2}} \right)^{-2} \right) \right|
    \lesssim \frac{1}{\sqrt{K}}.
\end{equation}

Finally, by combining~\eqref{Eq:W_diff2}, \eqref{Eq:bd_I2}, and~\eqref{Eq:bd_I1}, we proved that $W_{\mathrm{MC}}(c)-W_K(c)\lesssim \frac{1}{\sqrt{K}}$. Consequently, by \eqref{Eq:V_diff}, we have $V_{\mathrm{MC}}-V_K\lesssim \frac{1}{\sqrt{K}}$.

\end{document}